\documentclass[11pt,a4paper]{article}
\usepackage{lmodern}
\usepackage[utf8]{inputenc}
\usepackage[french,english]{babel}
\usepackage[T1]{fontenc}
\usepackage{amsmath, amsfonts, amssymb, amsthm}
\usepackage{fullpage}
\usepackage{microtype}
\usepackage{enumitem}
\setenumerate{label=(\alph*)}
\usepackage{mathtools} 
\usepackage{stmaryrd}
\usepackage{tikz}
\usepackage[smalltableaux=true]{ytableau}
\usepackage[colorlinks=true, linkcolor=blue]{hyperref}

\numberwithin{equation}{section}

\newtheorem{theorem}[equation]{Theorem}
\newtheorem{theoremet}{Theorem}

\newtheorem{proposition}[equation]{Proposition}

\newtheorem{corollary}[equation]{Corollary}
\newtheorem{lemma}[equation]{Lemma}

\theoremstyle{definition}
\newtheorem{definition}[equation]{Definition}

\theoremstyle{remark}
\newtheorem{remark}[equation]{Remark}
\newtheorem{example}[equation]{Example}

\allowdisplaybreaks[3]

\DeclareMathOperator\cosec{\mathrm{cosec}}
\newcommand\Z{\mathbb{Z}}
\newcommand\N{\mathbb{Z}_{\geq 0}}
\newcommand\Ze[1][e]{\Z/#1\Z}
\newcommand\R{\mathbb{R}}

\newcommand\E{\mathbb{E}}
\newcommand\Var{\mathrm{Var}}
\newcommand\kernelJ{\mathsf{J}^t}
\newcommand\kJ{\mathcal{J}}
\newcommand\besselJ[1][J]{{#1}^t}

\newcommand\besselsquare[1]{J_{#1}^{t,2}}

\newcommand\vare{\mathsf{e}}

\newcommand\ind{\mathbf{1}}
\newcommand\xrestr{\hat x}

\newcommand\ic{\mathbf{i}} 
\newcommand\iet{i^*}
\newcommand\half[1][1]{{\textstyle\frac{#1}{2}}}
\newcommand\parts{\mathcal{P}}
\newcommand\partsn[1][n]{\parts_{#1}}
\newcommand\tomega{\widetilde\omega}

\newcommand\valsumexp[1][\ell]{\mathfrak{e}_{#1}}

\newcommand\hdek[1][\ell]{\homega_{#1}}
\newcommand\ep{e'}

\newcommand\besselexp{\mathsf{F}}
\newcommand\sgn{\mathrm{sgn}}
\newcommand\sgnx[2]{\epsilon_{#1,#2}}
\newcommand\sumexp[1][j-i]{E^{[e]}_{#1}}
\newcommand\sumpartexp{\mathcal{E}}
\newcommand\sumexpzero{E^{[e]}}
\newcommand\homega{\hat\omega}
\newcommand\limd[1][k]{d_{e,#1}}

\newcommand\pl[1][n]{\mathrm{Pl}_{#1}}
\newcommand\plt[1][t]{\mathrm{pl}_{#1}}
\newcommand\rhot[1][t]{\rho^{#1}}
\newcommand\Et[1][t]{\E_{#1}}
\newcommand\Vart[1][t]{\Var_{#1}}
\newcommand\Covt[1][t]{\mathrm{Cov}_t}

\newcommand\D{\mathcal{D}}

\newcommand\rhod{\rho}

\newcommand\ci[1][i]{c_{#1}}
\newcommand\myxi[1][i]{x_{#1}}

\newcommand\cm\omega 

\newcommand\myvart{v_t}

\newcommand\permat{P}

\newcommand\myepsilon\delta

\newcommand\Y{\mathcal{Y}}

\author{Salim \textsc{Rostam}\thanks{Univ Rennes, CNRS, IRMAR - UMR 6625, F-35000 Rennes, France}}
\title{Core size of a random partition for the Plancherel measure}
\date{}

\begin{document}
\maketitle

\abstract{We prove that the size of the $e$-core of a partition chosen under the Poissonised Plancherel measure converges in distribution  to, as the Poisson parameter goes to $+\infty$ and after a suitable renormalisation, a sum of $e-1$ mutually independent Gamma distributions with explicit parameters. Such a result already exists for the uniform measure on the set of partitions of $n$ as $n$ goes to $+\infty$,  the parameters of the Gamma distributions being all equal. We rely on the fact that the descent set of a partition is a determinantal point process under the Poissonised Plancherel measure and on a central limit theorem for such processes.}

\section{Introduction}

The conjugacy classes of the symmetric group $\mathfrak{S}_n$ on $n$ letters are parametrised by the set $\partsn$ of \emph{partitions} of the integer $n$. These are the sequences $\lambda = (\lambda_1 \geq \lambda_2 \geq \dots)$ of positive integers with sum $|\lambda| \coloneqq n$.
 In particular, the irreducible complex representations $S^\lambda$ of $\mathfrak{S}_n$ can be parametrised by $\lambda \in \partsn$. It turns out that one can do this parametrisation such that the dimension of $S^\lambda$ as a $\mathbb{C}$-vector space is the number $\#\mathrm{Std}(\lambda)$ of \emph{standard Young tableaux} of shape $\lambda$, so that one has the formula
\begin{equation}
\label{equation:intro_n!}
n! = \sum_{\lambda \in \partsn} \#\mathrm{Std}(\lambda)^2
\end{equation}
(see, for instance, \cite{sagan}).
Quotienting by $n!$, the above equality exhibits a particular probability measure $\pl$ on $\partsn$, known as the \emph{Plancherel measure}. The equality~\eqref{equation:intro_n!} can also be seen as a consequence of the Robinson--Schensted correspondence between permutations in $\mathfrak{S}_n$ and pairs of standard Young tableaux of the same shape with $n$ boxes, for which the Plancherel measure $\pl$ is then the image of the uniform measure of $\mathfrak{S}_n$. This correspondence has many other additional properties, namely: if $\lambda = (\lambda_1 \geq \dots ) \in \partsn$ corresponds to $\sigma \in \mathfrak{S}_n$, then the first part $\lambda_1$  is exactly the length of a longest increasing subsequence in $\sigma$ (see, for instance, \cite{romik,sagan} for more details).

A natural question is the following: what does the probability measure $\pl$ on $\partsn$ look like? The question for an arbitrary $n$ is quite complicated; however there is a very nice description as $n \to +\infty$. Namely, Logan--Shepp~\cite{logan-shepp:variational} and independently Kerov--Vershik~\cite{kerov-vershik:asymptotics} proved that there exists a universal limit shape $\Omega$ for the partitions chosen under  $\pl$. More precisely, the rim of a partition $\lambda \in \partsn$ converges uniformly in probability under $\pl$ to the limit shape $\Omega$ as $n \to +\infty$. We represent in Figure~\ref{figure:lgn} such a large partition, drawn with the Russian convention for the Young diagram, together with the limit shape\footnote{Figures~\ref{figure:lgn}, \ref{figure:e4} and~\ref{figure:e7} were obtained using \textsc{SageMath}~\cite{sage}.}. Note that there exist limit shapes for other probability measures on $\partsn$, namely, the uniform measure (\cite{vershik}, with many other examples inside), the Gelfand measure \cite{logan-shepp:variational,meliot:gelfand,meliot:kerov} and the family of Schur--Weyl measures \cite{biane,feray-meliot}, the latter being a generalisation of the Plancherel measure. These limit shape results, which are analogues of the law of large numbers, have analogues of the central limit theorem. The first result of this kind was given by Kerov~\cite{kerov:gaussian} for the Plancherel measure: the difference between the partition and the limit shape converges in distribution  to a generalised Gaussian process (see, for instance, \cite{ivanov-olshanski,feray-meliot,meliot:kerov,bogachev-su:clt} for further results on this direction).

\begin{figure}
\begin{center}
\includegraphics[scale=.8]{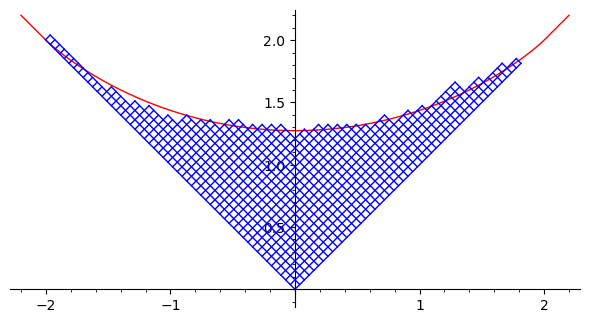}
\end{center}
\caption{Universal limit shape for the partitions under the Plancherel measure (here with $n = 700$).}
\label{figure:lgn}
\end{figure}

We now focus on the Plancherel measure $\partsn$. The derivative of the limit shape $\Omega$ represents the proportion of ``up'' and ``down'' steps made by the Young diagram at a given position. The ``down'' steps correspond to the elements of the \emph{descent set}
\[
\D(\lambda) \coloneqq \{\lambda_a - a : a \geq 1\} \subseteq \Z.
\]
Now if we look at several positions at the same time, for a given (finite) set $X \subseteq \Z$ we are interested in $\pl\bigl(X \subseteq \D(\lambda)\bigr)$. These quantities as functions of $X$ are known as \emph{correlation functions}. It turns out that these correlation functions are well-understood for the \emph{Poissonised} Plancherel measures on $\parts = \amalg_{n \geq 0} \partsn$. Given $t > 0$, the Poissonised Plancherel measure $\plt$ is the Plancherel measure $\pl$ where $n$ is a Poisson random variable of parameter $t$. In other words, setting $\pl(\lambda) = 0$ if $|\lambda| \neq n$, for any $\lambda \in \parts$ we have
\[
\plt(\lambda) = \exp(-t)\sum_{n = 0}^{+\infty}\frac{t^n}{n!}\pl(\lambda).
\]
(We denote by $\exp$ the exponential function, to avoid  any ambiguity with an integer $e \geq 2$ that we will intensively use in the sequel.) Borodin, Okounkov and Olshanski~\cite{boo} proved that for any fixed $s \geq 1$ and $x_1, \dots, x_s \in \Z$ pairwise distinct, we have
\begin{equation}
\label{equation:intro_BOO}
\plt\bigl(x_1,\dots,x_s \in \D(\lambda)\bigr) = \det\bigl[ \kernelJ(x_a,x_b)\bigr]_{1 \leq a,b \leq s},
\end{equation}
where $\kernelJ$ is the \emph{discrete Bessel kernel}, built with Bessel functions of the first kind. This means that the discrete point process $\D(\lambda)$ is \emph{determinantal}, with kernel $\kernelJ$. Under some conditions on $x_1,\dots,x_s$, they were able to \emph{de-Poissonise} the result, proving that the limit of $\pl\bigl(x_1,\dots,x_s \in \D(\lambda)\bigr)$ as $n \to +\infty$ satisfies an equality like~\eqref{equation:intro_BOO}, where the discrete Bessel kernel  $\kernelJ$ is replaced by the \emph{discrete sine kernel}.

Determinantal point processes are well-studied and appear for instance in the theory of random matrices and in mathematical physics (see, for instance,~\cite[\textsection 2]{soshnikov:survey}). Given a determinantal point process,  for each $t > 0$ let $I_t \subseteq \R$ be an interval and $\#_t$ the number of points of the determinantal point process inside $I_t$. Then under some conditions we have the following central limit theorem:
\[
\frac{\#_t - \E \#_t}{\sqrt{\Var \#_t}} \xrightarrow[t \to +\infty]{d} \mathcal{N}(0,1),
\]
where $\mathcal{N}(0,1)$ is the standard Gaussian distribution. This convergence in distribution was first established by Costin--Lebowitz~\cite{costin-lebowitz} for the sine kernel, and was then generalised for more general kernels by Widom and Soshnikov~\cite{soshnikov:gaussian, soshnikov:survey}. In the context of the determinantal point process $\D(\lambda)$, this central limit theorem was namely used by Bogachev and Su~\cite{bogachev-su:clt, bogachev-su} to prove a pointwise analogue of Kerov's central limit theorem~\cite{kerov:gaussian}.

\bigskip

\ytableausetup{nosmalltableaux}

\begin{figure}
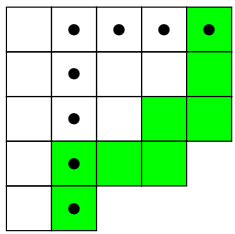

\begin{center}
\begin{ytableau}
{}  &\bullet &\bullet &\bullet &*(green)\bullet
\\
{} &\bullet&{} &{} &*(green)
\\
{} &\bullet &{} &*(green)&*(green)
\\
{} &*(green)\bullet &*(green)&*(green)
\\
{}&*(green)\bullet
\end{ytableau}
\end{center}
\caption{The hook $h_{(a,b)}$ and the rim hook $r_{(a,b)}$ for $(a,b) = (1,2)$ and 
$\lambda = (5,5,5,4,2)$.}
\label{figure:hook_rim}
\end{figure}

\ytableausetup{smalltableaux}

 Let~$\lambda = (\lambda_1 \geq \dots \geq \lambda_N > 0)$ be a partition. 
We will now recall a particularly nice link between the descent set $\D(\lambda)$ and the ``(rim) hook removal operation'' in the Young diagram of~$\lambda$ (we refer, for instance, to~\cite[\textsection 2.7]{james-kerber:representation} for more details).
 We identify the Young diagram $\Y(\lambda)$ with the corresponding part of $\Z_{\geq 1}^2$, more precisely:
\[
\Y(\lambda) \coloneqq \bigl\{(a,b) \in \Z_{\geq 1}^2 : 1 \leq a \leq N \text{ and } 1 \leq b \leq \lambda_a\bigr\}.
\]
 Given a node $(a,b) \in \Y(\lambda)$ of the Young diagram, let us define two subsets of $\Y(\lambda)$:
 \begin{itemize}
 \item the \emph{hook} $h_{(a,b)}$ is  given by:
\[
h_{(a,b)} \coloneqq \bigl\{(a',b') \in \Y(\lambda) : (a' = a \text{ and } b' \geq b) \text{ or } (a' \geq a \text{ and } b' = b)\bigr\},
\]
\item the \emph{rim hook} $r_{(a,b)}$ is the subset of the rim of $\Y(\lambda)$ given by  the nodes that are between the southmost and eastmost nodes of $h_{(a,b)}$, that is:
\[
r_{(a,b)} \coloneqq \bigl\{(a',b') \in \Y(\lambda) : a' \geq a,\, b' \geq b \text{ and } (a'+1,b'+1) \notin \Y(\lambda)\bigr\}.
\]
\end{itemize} 
An example of a hook and its associated rim hook is given in Figure~\ref{figure:hook_rim}. Note that we always have $\#h_{(a,b)} = \#r_{(a,b)}$.
Removing the hook $h_{(a,b)}$ from $\Y(\lambda)$ and then patching the two (possibly one) disconnect parts together gives another Young diagram, of a partition $\mu$. Note that $\Y(\mu)$ is also obtained from $\Y(\lambda)$ by removing the rim hook $r_{(a,b)}$. By definition of the hook removal operation, the following property holds:
\begin{center}
if $h \coloneqq \#r_{(a,b)}$ then there exists $\beta \in \D(\lambda)$ with $\beta-h \notin \D(\lambda)$
\\
such that $\D(\mu) = \bigl(\D(\lambda) \setminus\{\beta\}\bigr) \sqcup \{\beta-h\}$.
\end{center}
A rim hook with $\#r_{(a,b)} = h$ is called a \emph{$h$-rim hook}.

\medskip
Now fix $e \in \Z_{\geq 2}$. We are ready to define the notion of \emph{$e$-core} of a partition (again, we refer to~\cite[\textsection 2.7]{james-kerber:representation} for more details). Starting from (the Young diagram of) $\lambda$ we can repeatedly remove $e$-rim hooks until there are no more $e$-rim hooks. By the above property, the following two quantities are uniquely determined:
\begin{itemize}
\item the number $\mathrm{w}_e(\lambda)$ of $e$-rim hooks that we have removed,
\item the final partition $\overline{\lambda}$.
\end{itemize}
The integer $\mathrm{w}_e(\lambda)$ is the \emph{$e$-weight} of $\lambda$ and $\overline{\lambda}$ is the \emph{$e$-core} of $\lambda$. We show in Figure~\ref{figure:core} an example of an $8$-core calculation. Saying that two partitions $\lambda,\mu \in \partsn$ have the same $e$-core is equivalent to saying that the Young diagrams of $\lambda$ and $\mu$ have the same number $\ci(\lambda)$ of \emph{$i$-nodes} for all $i \in \Ze$.  It is also equivalent to~$\lambda$ and~$\mu$ belonging to the same block of a certain Iwahori--Hecke algebra (see, for instance, \cite[Chapter 5]{mathas:iwahori-hecke}).

\ytableausetup{nosmalltableaux}

\begin{figure}
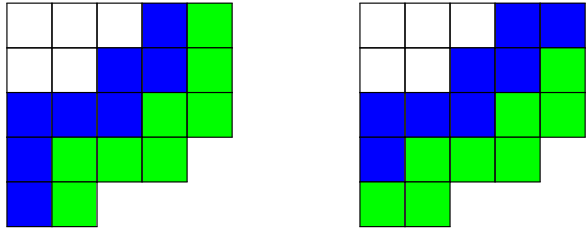

\begin{center}
\begin{ytableau}
{}  & & & *(blue)&*(green)
\\
{} & &*(blue) &*(blue) &*(green)
\\
*(blue) &*(blue) & *(blue)&*(green)&*(green)
\\
*(blue) &*(green) &*(green)&*(green)
\\
*(blue)&*(green)
\end{ytableau}
\qquad\qquad
\begin{ytableau}
{}  & & & *(blue)&*(blue)
\\
{} & &*(blue) &*(blue) &*(green)
\\
*(blue) &*(blue) & *(blue)&*(green)&*(green)
\\
*(blue) &*(green) &*(green)&*(green)
\\
*(green)&*(green)
\end{ytableau}
\end{center}
\caption{The $8$-core of
$\lambda = (5,5,5,4,2)$ is $\overline{\lambda} = (3,2).$}
\label{figure:core}
\end{figure}

\ytableausetup{smalltableaux}


\bigskip

If a limit shape exists for a given probability measure on $\partsn$ for $n \to +\infty$, it is thus natural to ask what is the behaviour of $\overline{\lambda}$ as $n \to +\infty$. When $\partsn$ is endowed with the uniform measure, Lulov--Pittel~\cite{lulov-pittel} and  later Ayyer--Sinha~\cite{ayyer-sinha} proved that $\frac{\pi}{\sqrt n}|\overline{\lambda}|$ converges in distribution to the Gamma distribution $\Gamma\bigl(\frac{e-1}{2},\sqrt 6\bigr)$ with shape $\frac{e-1}{2}$ and scale $\sqrt 6$.
In this paper, we study the behaviour of $|\overline\lambda|$ as $t \to +\infty$ when $\lambda$ is chosen under  the Poissonised Plancherel measure $\plt$. Note that the study of $|\overline{\lambda}|$ is a first step toward a better understanding of $\overline{\lambda}$, namely, the behaviour of its parts (see Remark~\ref{remark:length_first_row} for the size of the first part) or the behaviour of its rim (as in~\cite{lulov-pittel}).  As in~\cite{lulov-pittel,ayyer-sinha} we rely on the following identity:
\[
|\overline{\lambda}| = \frac{e}{2}\sum_{i \in \Ze} \myxi(\lambda)^2 + \sum_{i = 0}^{e-1} i \myxi(\lambda),
\]
where $\myxi(\lambda) \coloneqq \ci(\lambda) - \ci[i+1](\lambda)$.
Now the works of~\cite{lulov-pittel,ayyer-sinha} both rely on the asymptotics of Hardy--Ramanujan for the partition function $\#\partsn$. For our means, as in~\cite{bogachev-su:clt,bogachev-su} we heavily rely on the fact~\eqref{equation:intro_BOO} that $\D(\lambda)$ is determinantal~\cite{boo}  and we then exploit the first order asymptotics of Bessel functions. Our first result is the following (see Proposition~\ref{proposition:Et_bounded} and Theorem~\ref{theorem:limit_covt}), where $\Et$, $\Vart$ and $\Covt$ respectively denote the expectation,  variance and covariance with respect to $\plt$.

\begin{theoremet}
\label{theoremet:et_covt}
Under the Poissonised Plancherel measure $\plt$, as $t \to +\infty$ we have:
\begin{align*}
\Et \myxi(\lambda) &= O(1),
\\
\Covt\bigl(\myxi(\lambda),\myxi[j](\lambda)\bigr) &\sim \frac{2\sqrt t}{\pi e^2}\left[\cot(j-i+\half)\tfrac{\pi}{e} - \cot(j-i-\half)\tfrac{\pi}{e}\right],
\end{align*}
for all $i,j \in \Ze$. In particular, we have $\Vart \myxi(\lambda) \sim \frac{4\sqrt t}{\pi e^2}\cot\frac{\pi}{2e}$.
\end{theoremet}

We can already deduce the following asymptotics (Proposition~\ref{proposition:limit_Et_core}):
\[
\Et|\overline{\lambda}| \sim \frac{2\sqrt t}{\pi}\cot \frac{\pi}{2e},
\]
as $t \to +\infty$. To compare with~\cite{lulov-pittel,ayyer-sinha}, when $\lambda$ in chosen under the uniform measure on $\partsn$ then $\E |\overline{\lambda}| \sim \frac{e-1}{2\pi}\sqrt{6n}$ as $n \to +\infty$.
We then show that we can apply the central limit theorem~\cite{soshnikov:gaussian} with the variables $\myxi(\lambda)$ for $i \in \Ze$ (Theorem~\ref{theorem:convergence_vector_xi}).

\begin{theoremet}
\label{theoremet:vector}
The random vector
\[
e\sqrt{\frac{\pi}{2}}\left(\frac{\myxi(\lambda)}{t^{1/4}}\right)_{i \in \Ze},
\]
converges in distribution as $t \to +\infty$ to the centred normal distribution with covariance matrix $\bigl(\cot(j-i+\half)\tfrac{\pi}{e} - \cot(j-i-\half)\tfrac{\pi}{e}\bigr)_{i,j \in \Ze}$.
\end{theoremet}
 In particular, as $t \to +\infty$ the variables $t^{-1/4}\myxi(\lambda)$ and $t^{-1/4}\myxi[j](\lambda)$ for $i \neq j$ are not independent. We show in fact that the variables $\sum_{i = 0}^{e-1} \zeta^i t^{-1/4}\myxi(\lambda)$ for $\zeta \neq 1$ a complex $e$-th root of unity are (Gaussian and) mutually independent as $t \to +\infty$ .
We then deduce our main result, Theorem~\ref{theorem:size_core_plancherel}.

\begin{theoremet}
\label{theoremet:core}
Under the Poissonised Plancherel measure $\plt$, as $t \to +\infty$ the rescaled size $\frac{\pi}{4\sqrt t}|\overline\lambda|$ of the $e$-core of $\lambda$ converges  in distribution to a sum of mutually independent Gamma distributions $\Gamma\bigl(\frac{1}{2},\sin \frac{k\pi}{e}\bigr)$ for $k \in \{1,\dots,e-1\}$.
\end{theoremet}

 The difference between our main theorem and the result of~\cite{lulov-pittel,ayyer-sinha} is twofold:
\begin{itemize}
\item our convergence is for the Poissonised Plancherel measure $\plt$ on $\parts$  as $t \to +\infty$, and the not for the uniform measure on $\partsn$ as $n \to +\infty$,
\item the Gamma distribution $\Gamma\bigl(\frac{e-1}{2},\sqrt 6\bigr) = \sum_{k = 1}^{e-1} \Gamma\bigl(\frac{1}{2},\sqrt 6\bigr)$ is for us split into $e-1$ summands with different scale parameters.
\end{itemize}
Note that the order $\sqrt t$ of $|\overline\lambda|$ corresponds to the order $\sqrt n$ of~\cite{lulov-pittel,ayyer-sinha}, at it was already the case for the expectation. As in~\cite{boo,johansson} it should also be possible to state Theorem~\ref{theoremet:core} for the Plancherel measure $\pl$ for $n \to +\infty$. This process, the so-called \emph{de-Poissonisation}, will not be considered here. Note that, as it is mentioned in~\cite{costin-lebowitz,soshnikov:gaussian}, the convergence in the central limit theorem holds in fact in moments, so that we deduce
\[
\Vart|\overline{\lambda}| \sim  \frac{4et}{\pi^2},
\]
as $t \to +\infty$ (Corollary~\ref{corollary:limit_variance}).

\medskip

Let us conclude with two variations on our problem. First, as we already mentioned, the Plancherel measure
on partitions is connected with other probabilistic objects such as longest subsequences or
random matrices. To the author's knowledge, it is not known whether the $e$-core of partitions under
the Plancherel measure behaves well with these connections.
Second, besides the asymptotic behaviour of the $e$-core under the (Poissonised) Plancherel measure, one can also study the \emph{$e$-quotient} (see, for instance,~\cite[\textsection 2.7]{james-kerber:representation}). The $e$-quotient of a partition is an $e$-tuple of partitions, whose construction is complementary to the $e$-core of the partition; namely, the data of the $e$-core and the $e$-quotient allows to recover the (unique) underlying partition. The same questions as for the $e$-core can be studied (\textit{e.g.} size and size of the parts) but it should also be interesting to study the behaviour of the descent sets, namely, whether they are determinantal.

\bigskip

We now give the outline of the paper. In Section~\ref{section:setting} we recall the result from Borodin-Okounkov-Olshanski saying that $\D(\lambda)$ is a determinantal point process under $\plt$.
Section~\ref{section:cores} is devoted to the combinatorics of $e$-cores. In Definition~\ref{definition:xi} we define the variables $\myxi(\lambda) \coloneqq \ci(\lambda) - \ci[i+1](\lambda)$, which we relate in~\eqref{equation:size_core} to $|\overline\lambda|$ by a quadratic polynomial and in  Lemma~\ref{lemma:x_i_card_Fr} to $\D(\lambda)$ in a view to applying the central limit theorem of Costin--Lebowitz and Soshnikov for determinantal point process.
Section~\ref{section:covariance}  is devoted to expectation and covariance calculations. In~\textsection\ref{subsection:clt} we show how we will apply the central limit theorem to $\myxi(\lambda)$.
In~\textsection\ref{subsection:expectation} we prove that the expectation $\Et\myxi(\lambda)$ of $\myxi(\lambda)$ under the Poissonised Plancherel measure $\plt$ is bounded (Proposition~\ref{proposition:Et_bounded}).
 Subsection~\textsection\ref{subsection:asymptotics_covariance} is the technical heart of the paper. We compute the first order asymptotics of the covariance $\Covt\bigl(\myxi(\lambda),\myxi[j](\lambda)\bigr)$  under $\plt$ for $i \neq j \in \Ze$ (Theorem~\ref{theorem:limit_covt}). We deduce in~\textsection\ref{subsection:asymptotics_variance} the first order asymptotics of the variance $\Vart \myxi(\lambda)$ under $\plt$ (Corollary~\ref{corollary:asymptotics_vart}).  In Section~\ref{section:size_core} we apply the preceding results to study the asymptotics of $|\overline{\lambda}|$ when $\lambda$ is chosen under the Poissonised Plancherel measure $\plt$ as $t \to +\infty$. We first obtain in Proposition~\ref{proposition:limit_Et_core} the asymptotics of the expectation $\Et|\overline{\lambda}|$. Using the central limit theorem of Costin--Lebowitz and Soshnikov, we obtain in Theorem~\ref{theorem:convergence_vector_xi} that the random vector $\bigl(\myxi(\lambda)\bigr)_{i \in \Ze}$ converges in distribution to a certain centred normal vector. Then using a contour integration to compute the eigenvalues of the covariance matrix, we deduce our main result, Theorem~\ref{theorem:size_core_plancherel}, which says  that $|\overline{\lambda}| / \sqrt t$ converges in distribution under $\plt$ as $t \to +\infty$ to a sum of squares of mutually independent centred Gaussian variables, that is, a sum of mutually independent Gamma variables of shape $\frac{1}{2}$. Since the convergence in the central limit theorem holds in moments, we deduce in Corollary~\ref{corollary:limit_variance} the asymptotics of the variance $\Vart|\overline{\lambda}|$.
 Finally, in Section~\ref{section:inodes} we apply the preceding results to compute the first order asymptotics of $\ci(\lambda)$ under the Poissonised Plancherel measure $\plt$ (Proposition~\ref{proposition:ci/t}). It turns out that we can recover this result in the (non-Poissonised) Plancherel  setting (Proposition~\ref{proposition:ci/n}) using the law of large numbers of Logan--Shepp and Kerov--Vershik.

\paragraph{Acknowledgements}
The author is thankful to Jean-Christophe Breton, Valentin Féray, Cédric Lecouvey and Pierre-Loïc Méliot  for many useful discussions. The author is supported by the Agence Nationale de la Recherche funding ANR CORTIPOM 21-CE40-0019. The author also thanks the Centre Henri Lebesgue ANR-11-LABX-0020-0.

\section{Setting}
\label{section:setting}

Let $n \in \Z_{\geq 0}$.

\subsection{Partitions}
\label{subsection:young}

\paragraph{Partitions}
A \emph{partition} of $n$ is a non-increasing sequence $\lambda = (\lambda_1 \geq \dots \geq \lambda_{h(\lambda)} > 0)$ of positive integers with sum $|\lambda| \coloneqq \sum_{a = 1}^{h(\lambda)} \lambda_a = n$, where $h(\lambda) \geq 0$. We denote by $\partsn$ the set of partitions of $n$ and $\parts \coloneqq \sqcup_{n \geq 0} \partsn$. The \emph{conjugate} of $\lambda$, denoted by~$\lambda'$, is the partition given by $\lambda'_a = \#\{b \geq 1 : \lambda_b \geq a\}$ for all $a \in \{1,\dots,\lambda_1\}$.

\paragraph{Young diagrams}

The \emph{Young diagram} of $\lambda\in\partsn$ is the subset $\mathcal{Y}(\lambda)$ of $\Z_{\geq 1}^2$ given by
\[
\mathcal{Y}(\lambda) \coloneqq \left\{(a,b) \in \Z_{\geq 1}^2 : 1 \leq a \leq h(\lambda) \text{ and } 1 \leq b \leq \lambda_a\right\}.
\]

\begin{example}
The Young diagram associated with the partition $(4,3,2,2)$ is $\ydiagram{4,3,2,2}$ .
\end{example}
 The Young diagram of $\lambda'$ is obtained by flipping $\mathcal{Y}(\lambda)$ with respect to the diagonal, in other words, we have $(a,b) \in \mathcal{Y}(\lambda') \iff (b,a) \in \mathcal{Y}(\lambda)$.

 A \emph{standard tableau of shape $\lambda$} is a bijection $\mathfrak{t} : \mathcal{Y}(\lambda) \to \{1,\dots,n\}$ such that $\mathfrak{t}$ increases along the rows and down the columns, in other words for $(a,b) \in \mathcal{Y}(\lambda)$ we have $\mathfrak{t}(a,b) < \mathfrak{t}(a+1,b)$ if $(a+1,b) \in \mathcal{Y} (\lambda)$ and $\mathfrak{t}(a,b) < \mathfrak{t}(a,b+1)$ if $(a,b+1) \in \mathcal{Y}(\lambda)$. We denote by $\mathrm{Std}(\lambda)$ the set of standard Young tableaux of shape $\lambda$. Note that $\mathrm{Std}(\lambda)$ and $\mathrm{Std}(\lambda')$ are naturally in bijection.

\paragraph{Russian convention}
Rotating the Young diagram of $\lambda'$ by an angle of $\frac{3\pi}{4}$ and applying a linear homothety of ratio $\sqrt 2$ gives the \emph{Russian convention} for the Young diagram of $\lambda$. Note that the node $(a,b) \in \mathcal{Y}(\lambda)$ corresponds to the (square) box with opposite vertices $(a-1,b-1)$ and $(a,b)$.
 We denote by $\omega_\lambda : \mathbb{R} \to \mathbb{R}$ the upper rim of the resulting diagram, extending $\omega_\lambda$ by $\omega_\lambda(x) \coloneqq |x|$ outside the diagram.
Then $\omega_\lambda$ is a  continuous piecewise linear function such that:
\begin{itemize}
\item for each $k \in \Z$ we have
\begin{equation}
\label{equation:omega'}
\omega'_\lambda|_{(k,k+1)} = \pm 1,
\end{equation}
\item we have $\omega_\lambda(x) = |x|$ for $|x| \gg 0$ (more precisely, for $x \leq -\lambda'_1$ or $x \geq \lambda_1$),
\item we have $\int_{\mathbb{R}} \bigl[\omega_\lambda(x) - |x|\bigr] dx = 2n$, in particular each box in the Young diagram of $\lambda$ in the Russian convention has area $2$ and semi-diagonal length $1$.
\end{itemize}
An illustration of the construction of $\omega_\lambda$ is given in Figure~\ref{figure:omegalambda}.
Note that for any $x \in \mathbb{R}$ we have
\begin{equation}
\label{equation:omegalambda'}
\omega_{\lambda'}(x) = \omega_\lambda(-x).
\end{equation}

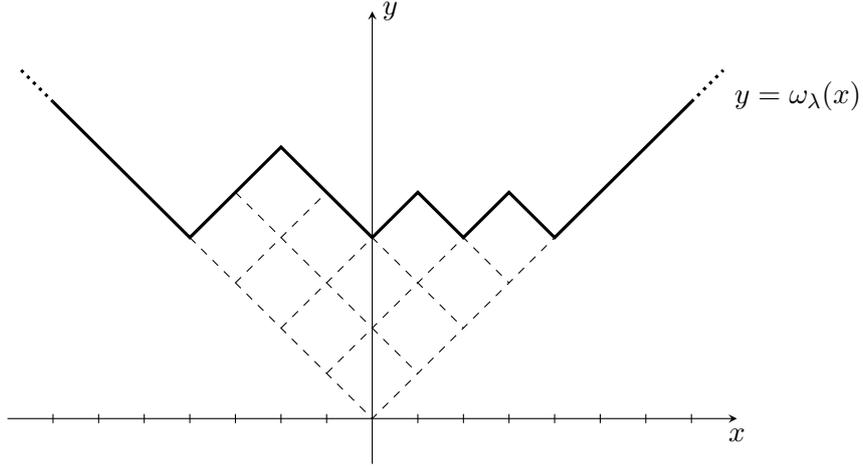
\begin{figure}
\begin{center}
\begin{tikzpicture}[scale=.6]
\draw[-stealth] (-8,0) -- (8,0) node[below]{$x$};
\draw[-stealth] (0,-1) -- (0,9) node[right]{$y$};
\foreach \i in {-7,-6,-5,-4,-3,-2,-1,1,2,3,4,5,6,7}
	{\draw (\i,-.1) --++ (0,.2);}

\draw[very thick] (-7,7) -- (-4,4) -- (-2,6) -- (0,4) -- (1,5) -- (2,4) -- (3,5) -- (4,4) -- (7,7);
\draw[dotted, very thick] (-7.7,7.7) -- (-7,7);
\draw[dotted, very thick] (7,7) -- (7.7,7.7) node[below right]{$y=\omega_\lambda(x)$};

\draw[dashed] (-4,4) -- (0,0) -- (4,4);
\draw[dashed] (-3,5) -- (0,2) -- (2,4) --++ (1,-1);
\draw[dashed] (-3,3) --++ (2,2);
\draw[dashed] (-2,2) --++ (2,2) --++ (2,-2);
\draw[dashed] (-1,1) --++ (1,1) --++ (1,-1);
\end{tikzpicture}
\end{center}
\caption{Russian convention for the Young diagram of $\lambda = (4,3,2,2)$}
\label{figure:omegalambda}
\end{figure}

\paragraph{Descent set}

Yet another way to look at a partition $\lambda$ is to consider its \emph{descent set}, or \emph{$\beta$-set} (in the terminology of~\cite{boo} and~\cite{olsson}, respectively), defined as
\[
\D(\lambda) = \bigl\{ \lambda_a - a : a \geq 1 \bigr\} \subseteq \Z,
\]
with $\lambda_a \coloneqq 0$ if $a > h(\lambda)$.  The set $\D(\lambda) \cap \Z_{\geq 0}$ is finite, while $\D(\lambda) \supseteq \Z_{< -h(\lambda)}$. The terminology \emph{descent set} is justified by the following:
\begin{equation}
\label{equation:D(lambda)_omegalambda}
k \in \D(\lambda) \iff \omega_\lambda(k) > \omega_\lambda(k+1),
\end{equation}
for any $k \in \Z$. If particular, using~\eqref{equation:omegalambda'} we recover the following classical relation:
\begin{equation}
\label{equation:D(lambda')}
\D(\lambda') = -\D(\lambda)^c -1.
\end{equation}
An example of descent set is given in Figure~\ref{figure:descent_set}.

\begin{figure}
\begin{center}
\begin{tikzpicture}[scale=.6]
\draw[-stealth] (-8,0) -- (8,0) node[below]{$x$};
\draw[-stealth] (0,-1) -- (0,9) node[right]{$y$};
\foreach \i in {-7,-6,-5,-4,-3,-2,-1,1,2,3,4,5,6,7}
	{\draw (\i,-.1) --++ (0,.2);}

\draw[very thick] (-7,7) -- (-4,4) -- (-2,6) -- (0,4) -- (1,5) -- (2,4) -- (3,5) -- (4,4) -- (7,7);
\draw[dotted, very thick] (-7.7,7.7) -- (-7,7);
\draw[dotted, very thick] (7,7) -- (7.7,7.7) node[below right]{$y=\omega_\lambda(x)$};

\draw[gray,dashed] (-4,4) -- (0,0) -- (4,4);
\draw[gray,dashed] (-3,5) -- (0,2) -- (2,4) --++ (1,-1);
\draw[gray,dashed] (-3,3) --++ (2,2);
\draw[gray,dashed] (-2,2) --++ (2,2) --++ (2,-2);
\draw[gray,dashed] (-1,1) --++ (1,1) --++ (1,-1);

\draw[dashed, blue] (3,5) -- (3,0);
\draw[dashed, blue] (1,5) -- (1,0);
\draw[dashed, blue] (-1,5) -- (-1,0);
\draw[dashed, blue] (-2,6) -- (-2,0);
\foreach \i in {-7,-6,-5}
	{\draw[dashed, blue] (\i,-\i) -- (\i,0);}

\foreach \i in {-7,-6,-5,-2,-1,1,3}
	{\fill[blue] (\i,0) circle (.1);}
\end{tikzpicture}
\end{center}
\caption{Descent set $\D(\lambda) = \{3,1,-1,-2,-5,-6,-7,\dots\}$ for $\lambda= (4,3,2,2)$}
\label{figure:descent_set}
\end{figure}
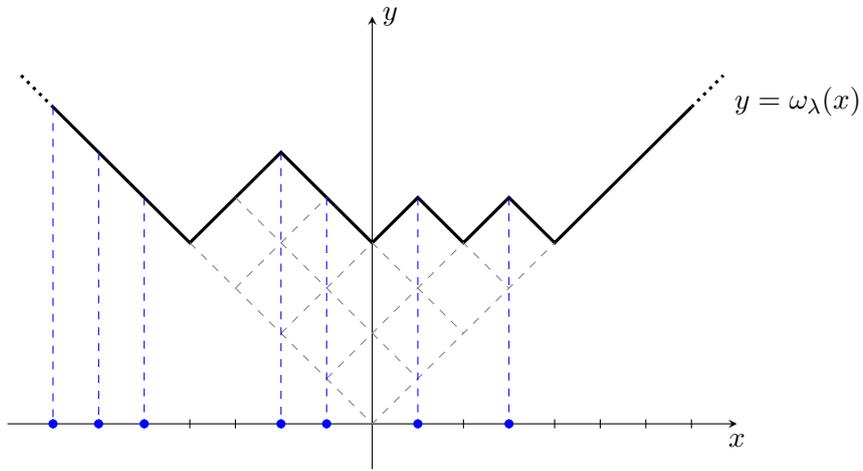

\subsection{Plancherel measure}

\paragraph{Plancherel measure}

Recall from the introduction the following identity:
\begin{equation}
\label{equation:burnside}
n! = \sum_{\lambda \in \partsn} \#\mathrm{Std}(\lambda)^2.
\end{equation}

\begin{definition}
The \emph{Plancherel measure} on the set $\partsn$ of the partitions of $n$ is given by
\[
\pl(\lambda) \coloneqq \frac{\#\mathrm{Std}(\lambda)^2}{n!},
\]
for all $\lambda \in \partsn$.
\end{definition}

Recalling that $\lambda'$ denotes the conjugate partition of $\lambda$, we have $\#\mathrm{Std}(\lambda) = \#\mathrm{Std}(\lambda')$ thus
\begin{equation}
\label{equation:pl(ld)=pl(ld')}
\pl(\lambda) = \pl(\lambda').
\end{equation}

\begin{definition}
Let $t > 0$. The \emph{Poissonised} Plancherel measure $\plt$ on the set $\parts$ of all partitions is given for $\lambda \in \parts$  by
\[
\plt(\lambda) \coloneqq \exp(-t) t^{\lvert \lambda\rvert} \left(\frac{\#\mathrm{Std}(\lambda)}{\lvert \lambda\rvert !}\right)^2.
\]
\end{definition}
In other words, setting $\pl(\lambda) \coloneqq 0$ if $|\lambda| \neq n$, we have
\[
\plt = \exp(-t) \sum_{n \geq 0} \frac{t^n}{n!}\pl.
\]
We use the notation $\exp$ for the exponential function since the notation $e$ will be used for an integer later (see~\textsection\ref{section:cores}).
It follows from~\eqref{equation:pl(ld)=pl(ld')} that if $\Lambda$ is any set of partitions and $\Lambda' \coloneqq \{\lambda' : \lambda \in \Lambda\}$ then
\begin{equation}
\label{equation:pl(Lambda)}
\plt(\Lambda) = \plt(\Lambda').
\end{equation}

\paragraph{Correlation function}

\begin{definition}
\label{definition:correlation_function}
The \emph{correlation function} $\rhot$ is defined for any finite subset $X  \subseteq \Z$ by
\[
\rhot(X) \coloneqq \plt\bigl\{\lambda \in \parts : X \subseteq \D(\lambda)\bigr\}.
\]
\end{definition}

In the sequel, when we write $X = \{x_1,\dots,x_s\} \subseteq \Z$ then it will always be understood that $x_a \neq x_b$ if $a \neq b$. In particular, we write
\[
\rhot(x_1,\dots,x_s) \coloneqq \rhot(X).
\]
Note that~\eqref{equation:pl(Lambda)} gives that we also have
\begin{equation}
\label{equation:corr_lambda'}
\rhod^t(X) = \plt\left\{\lambda \in \parts : X \subseteq \D(\lambda')\right\}.
\end{equation}

\begin{lemma}
\label{lemma:1-rho}
For any $k \in \Z$ we have
\[
1-\rhod^t(k)=\rhod^t(-k-1).
\]
\end{lemma}

\begin{proof}
By~\eqref{equation:D(lambda')} we have
\begin{align*}
1 - \rhod^t(k)
&=
1-\plt\left(k \in \D(\lambda)\right)
\\
&=
\plt\left(k \in \D(\lambda)^c\right)
\\
&=
\plt\left(-k-1 \in -\D(\lambda)^c-1\right)
\\
&=
\plt\left(-k-1 \in \D(\lambda')\right),
\end{align*}
and we conclude by~\eqref{equation:corr_lambda'}.
\end{proof}

\paragraph{Discrete Bessel kernel}
For any $x \in \R$, we denote by $J_x$ the Bessel function of order $x$ (see~\cite[\textsection 10]{nist}) and we define $\besselJ_x \coloneqq J_x(2\sqrt t)$. We will also write $\besselsquare{x} \coloneqq {(\besselJ_x)}^2$. 

\begin{definition}[\protect{\cite{boo,johansson}}]
The \emph{discrete Bessel kernel} is defined for any $x, y \in \R$ by
\[
\kernelJ(x,y)\coloneqq \sqrt t \frac{\besselJ_x \besselJ_{y+1} - \besselJ_{x+1} \besselJ_y}{ x- y}.
\]
\end{definition}
Note that the diagonal values $\kernelJ(x,x)$ are well-defined since $\besselJ_x$ is an entire function of $x$. For instance, we have
 \[
 \kernelJ(x,x) = 
\sqrt{t}\left[\besselJ[L]_x \besselJ_{x+1} - \besselJ_x \besselJ[L]_{x+1}\right],
\]
where $\besselJ[L]_x = \frac{\partial}{\partial x} \besselJ_x$ (see~\cite[(3.27)]{johansson}). The next result shows that, under $\plt$, the point process $\D(\lambda)$ is determinantal.

\begin{theorem}[\protect{\cite[Theorem 2]{boo}}]
\label{theorem:BOO}
For any $\{x_1,\dots,x_s\} \subseteq \Z$ we have
\[
\rhot(x_1,\dots,x_s) = \det \bigl[\kernelJ(x_a,x_b)\bigr]_{1\leq a,b\leq s}.
\]
\end{theorem}

\begin{remark}
Under $\pl$, the process $\D(\lambda)$ is a priori not determinantal; however~\cite{boo} proved namely that, under some assumptions,
\[
\pl\bigl\{\lambda \in \partsn : x_1(n),\dots,x_s(n) \in \D(\lambda)\bigr\} \xrightarrow{n \to +\infty}\det \bigl[\mathsf{S}(d_{ab},y)\bigr]_{1\leq a,b\leq s},
\]
where $\mathsf{S}$ is the \emph{discrete sine kernel}, where $y$ and $(d_{ab})_{ab}$ are determined by $\bigl(x_a(n)\bigr)_{a,n}$.
\end{remark}

We now recall some properties of the kernel $\kernelJ$. 

\begin{lemma}
The kernel $\kernelJ$ is symmetric: for all $x,y \in \Z$ we have $\kernelJ(x,y) = \kernelJ(y,x)$.
\end{lemma}

\begin{lemma}[\protect{\cite[Proposition 2.9]{boo}, \cite[(3.31)]{johansson}}]
\label{lemma:kernelJ_sum}
For all $x,y\in\Z$ we have
\[
\kernelJ(x,y) = \sum_{s \geq 1} \besselJ_{x+s}\besselJ_{y+s}.
\]
In particular,
\[
\kernelJ(x,x) = \sum_{s \geq 1} \besselsquare{x+s}.
\]
\end{lemma}

\section{Cores}
\label{section:cores}

Let  $e \in \Z_{\geq 2}$. We will identify $\Ze$ and $\{0,\dots,e-1\}$ in the natural way. Let $\lambda \in \partsn$ be a partition of $n$.

\begin{definition}
For any $i \in \Ze$, we define $\ci(\lambda) \in \Z_{\geq 0}$ to be the number of nodes $(a,b) \in \mathcal{Y}(\lambda)$ of the Young diagram of $\lambda$ such that $b-a \equiv i \pmod e$.
\end{definition}

Such a node $(a,b) \in \mathcal{Y}(\lambda)$ with $b-a \equiv i \pmod e$ is an \emph{$i$-node} of $\mathcal{Y}(\lambda)$. Note that
\begin{equation}
\label{equation:sum_ci}
\sum_{i \in \Ze} \ci(\lambda) = |\lambda|.
\end{equation}
The next result is immediate from the definition of $\omega_\lambda$.

\begin{lemma}
\label{lemma:ci_omegalambda}
For any $i \in \Ze$ we have
\[
\ci(\lambda) = \frac{1}{2}
\sum_{\substack{m \in \Z \\ m = i \bmod e}} \omega_\lambda(m) - |m|
=
\frac{1}{2}\sum_{k \in \Z} \omega_\lambda(i+ke) - \lvert i +ke\rvert.
\]
\end{lemma}

\begin{proof}
First note that the sum is finite. Recall that, in the Russian convention, each box has diagonal $2$. Then for any $m \in \Z$, the number $\frac{1}{2}\left[\omega_\lambda(m) - |m|\right]$ is (an integer and is) the number of nodes $(a,b) \in \mathcal{Y}(\lambda)$ such that $b-a = m$.
\end{proof}

If $\lambda,\mu \in \parts$ are such that $\ci(\lambda) = \ci(\mu)$ for all $i \in \Ze$ then we do not necessarily have $\lambda = \mu$ (unless if we are in the, excluded, case $e = 0$). However, as we mentioned in the introduction we know that $\lambda$ and $\mu$ share the same \emph{$e$-core} (see Remark~\ref{remark:same_xi} for a slight generalisation). As a quick reminder, the $e$-core of  $\lambda$, denoted by $\overline{\lambda}$, is constructed from $\mathcal{Y}(\lambda)$ by successively removing as many $e$-rim hooks as possible, where $e$-rim hooks are certain subsets of $e$ consecutive nodes of the rim of the Young diagram. The \emph{$e$-weight} $\mathrm{w}_e(\lambda)$ of $\lambda$ is the number of $e$-rim hooks that we  remove to reach the $e$-core $\overline{\lambda}$, in particular, we have the relation:
 \begin{equation}
 \label{equation:eweight}
 |\lambda| = |\overline{\lambda}| + e \mathrm{w}_e(\lambda).
 \end{equation}
More precisely,  since any $e$-rim hook has exactly one $i$-node for any $i \in \Ze$  (this follows from the hook removal operation; see also, for instance,~\cite[Remark 2.3]{rostam}) we deduce that
\begin{equation}
\label{equation:cilambda_cicore}
\ci(\lambda) = \ci(\overline{\lambda}) + \mathrm{w}_e(\lambda),
\end{equation}
 recovering~\eqref{equation:eweight} by~\eqref{equation:sum_ci}.

\begin{definition}
\label{definition:xi}
For any $i \in \Ze$ we define $\myxi(\lambda) \coloneqq \ci(\lambda) - \ci[i+1](\lambda) \in \Z$.
\end{definition}

Note that
\begin{equation}
\label{equation:sum_xi_0}
\sum_{i \in \Ze} x_i(\lambda) = 0.
\end{equation}

\begin{remark}
\label{remark:eabacus}
The quantities $x_0(\lambda),\dots,x_{e-1}(\lambda)$ can be read on the \emph{$e$-abacus} of $\overline\lambda$, see, for instance, \cite[Lemma 2.11]{rostam}. 
\end{remark}

The analogue of~\eqref{equation:sum_ci} is the following (non-trivial) equality, where $\lVert \cdot \rVert$ denotes the Euclidean norm:
\begin{equation}
\label{equation:c0lambdabar}
\ci[0](\overline{\lambda}) = \frac{1}{2}\lVert x(\lambda)\rVert^2.
\end{equation}
Indeed, by~\cite[Proposition 2.1]{fayers}  (see also~\cite[Proposition 2.14]{rostam} or also~\cite[Bijection 2]{gks}) we have $\ci[0](\overline{\lambda}) = \frac{1}{2}\lVert x(\overline{\lambda})\rVert^2$, but $x(\overline{\lambda}) = x(\lambda)$ by~\eqref{equation:cilambda_cicore}. Hence, for any $i \in \{1,\dots,e-1\}$ we have
\begin{equation}
\label{equation:cilambdabar}
\ci(\overline{\lambda}) = \frac{1}{2}\lVert x(\lambda)\rVert^2 - \myxi[0](\lambda) - \dots - \myxi[i-1](\lambda).
\end{equation}

\begin{remark}
\label{remark:same_xi}
It follows that if $\lambda,\mu \in \parts$ satisfy $\myxi(\lambda) = \myxi(\mu)$ for all $i \in \Ze$ then $\ci(\overline\lambda) = \ci(\overline\mu)$ for all $i \in \Ze$ thus $\overline{\lambda} = \overline{\mu}$. 
\end{remark}

Note the following consequence of~\eqref{equation:sum_ci}, \eqref{equation:sum_xi_0}, \eqref{equation:c0lambdabar} and \eqref{equation:cilambdabar}:
\begin{equation}
\label{equation:size_core}
\lvert \overline{\lambda}\rvert = \frac{e}{2}\lVert x(\lambda)\rVert^2 + \sum_{i =0}^{e-1}i x_i(\lambda).
\end{equation}

One of the aim of this paper is to study the asymptotic behaviour of $\lvert \overline{\lambda}\rvert$ under the Plancherel measure. By~\eqref{equation:size_core}, for any $i \in \Ze$ we have
\[
x_i(\lambda) = O\left(\sqrt {|\lambda|}\right).
\]
 We will see how to refine this asymptotics when $\lambda$ is chosen under the (Poissonised)  Plancherel measure, by using Theorem~\ref{theorem:BOO} (see Theorem~\ref{theorem:size_core_plancherel}). 

\begin{lemma}
\label{lemma:x_i_card_Fr}
For any $i \in \{0,\dots,e-1\}$ we have
\begin{align*}
x_i(\lambda)
&=
\# (e\Z_{\geq 0} + i) \cap \D(\lambda)   - \# (e\Z_{< 0} + i) \cap \D(\lambda)^c
\\
&=
\# (e\Z_{\geq 0} + i) \cap \D(\lambda)   -  \# (e\Z_{\geq 1} -i-1) \cap \D(\lambda')
\\
&=
\# (e\Z_{\geq 0} + i) \cap \D(\lambda)   -  \# (e\Z_{\geq 0} + \iet) \cap \D(\lambda'),
\end{align*}
with $\iet \coloneqq e-1-i \in \{0,\dots,e-1\}$.
\end{lemma}

\begin{proof}
Write $\delta_k \coloneqq \omega_\lambda(ke+i) - \omega_\lambda(ke+i+1)$. Recalling~\eqref{equation:omega'} and~\eqref{equation:D(lambda)_omegalambda}, we have:
\begin{equation}
\label{equation:delta_lambda_D}
\delta_k = \begin{cases}
1, &\text{if }ke+i \in \D(\lambda),
\\
-1,&\text{otherwise}.
\end{cases}
\end{equation}
 We have, using Lemma~\ref{lemma:ci_omegalambda} and the fact that $i \in \{0,\dots,e-1\}$,
\begin{align*}
2x_i(\lambda)
&=
2\bigl(\ci[i](\lambda) - \ci[i+1](\lambda)\bigr)
\\
&=
\sum_{k \in \Z} \Bigl(\omega_\lambda(ke+i) - \lvert ke+i\rvert\Bigr) - \Bigl(\omega_\lambda(ke+i+1) - \lvert ke+i+1\rvert\Bigr)
\\
&=
\sum_{k < 0} \delta_k - \lvert ke+i\rvert + \lvert ke+i+1\rvert
+
\sum_{k \geq 0} \delta_k - \lvert ke+i\rvert + \lvert ke+i+1\rvert
\\
&=
\sum_{k < 0} \delta_k + (ke+i) - (ke+i+1)
+
\sum_{k \geq 0} \delta_k - (ke+i) + (ke+i+1)
\\
&=
\sum_{k < 0} \bigl(\delta_k  -1\bigr)
+
\sum_{k \geq 0} \bigl(\delta_k +1\bigr).
\end{align*}
By~\eqref{equation:delta_lambda_D}, we have
\[
\delta_k + 1 = \begin{cases}
2, &\text{if } ke+i \in \D(\lambda),
\\
0,&\text{otherwise},
\end{cases}
\]
thus $\frac{1}{2}\sum_{k \geq 0}(\delta_k + 1) = \#(e\Z_{\geq 0} + i)\cap \D(\lambda)$. Similarly, we have
\[
\delta_k - 1 = \begin{cases}
-2, &\text{if } ke+i \in \D(\lambda)^c,
\\
0,&\text{otherwise},
\end{cases}
\]
thus $\frac{1}{2}\sum_{k < 0} (\delta_k - 1) = - \#(e\Z_{<0} + i)\cap \D(\lambda)^c$. We deduce the first equality of the lemma.  By~\eqref{equation:D(lambda')} we have $\#(e\Z_{<0} + i) \cap \D(\lambda)^c = \#(e\Z_{\geq 1} -i - 1)\cap \D(\lambda')$  thus we deduce the last two announced equalities.
\end{proof}

\section{Covariance calculations}
\label{section:covariance}

 Our aim is to evaluate $x_i(\lambda)$ as $\lvert \lambda \rvert \to \infty$ under the Poissonised Plancherel measure. 

\subsection{Central limit theorem}
\label{subsection:clt}

With a view to Lemma~\ref{lemma:x_i_card_Fr}, we will use the following central limit theorem.

\begin{theorem}[\protect{Costin--Lebowitz~\cite{costin-lebowitz}, Soshnikov~\cite{soshnikov:gaussian}}]
\label{theorem:clt}
Let $(\mathcal{D}_t)_t$ be a sequence of determinantal point processes on a locally compact Polish space. Assume that each $\mathcal{D}_t$ is associated with a Hermitian non-negative locally trace class operator in $L^2$. Let $\bigl\{I^{(0)}_t,\dots,I^{(e-1)}_t\bigr\}_t$ be a sequence of measurable sets, disjoints for any fixed $t$. We denote by $\Et$, $\Vart$ and $\Covt$ the expectation,  variance and covariance with respect to the probability distribution of the random process $\mathcal{D}_t$. Let $\#_t^{(i)}$ be the random variable given by the number of points of $\mathcal{D}_t$ inside $I_t^{(i)}$. If $\Vart \#_t^{(i)} \to +\infty$ as $t \to +\infty$ for all $i$ and if for all $i,j$
\[
\frac{\Covt\left(\#_t^{(i)}, \#_t^{(j)}\right)}{\sqrt{\Vart \#_t^{(i)} \Vart \#_t^{(j)}}} \xrightarrow{t \to +\infty} b_{ij},
\]
for some $b_{ij} \in \R$ then the vector
\[
\left(\frac{\#_t^{(i)} - \Et \#_t^{(i)}}{\sqrt{\Vart \#_t^{(i)}}}\right)_{0\leq i < e}
\]
converges  in distribution to the $k$-dimensional centred normal vector with covariance matrix $(b_{ij})_{0 \leq i,j < e}$.
\end{theorem}

\begin{remark}
Theorem~\ref{theorem:clt} in~\cite{soshnikov:gaussian} has only been explicitly stated  for some particular determinantal  processes. However, the proof is valid in the general case (see~\cite[Remark 2]{soshnikov:clt_stat} and~\cite[\textsection 4]{soshnikov:survey}). See also~\cite{gaf}.
\end{remark}

\begin{remark}
The discrete Bessel kernel $\kernelJ$ satisfies the assumptions of Theorem~\ref{theorem:clt}, see~\cite[Corollary 2.10]{boo}. In particular, we will be able to use Theorem~\ref{theorem:clt} in our setting.
\end{remark}

Let $i \in \{0,\dots,e-1\}$. Recall from Lemma~\ref{lemma:x_i_card_Fr} that $x_i(\lambda) = \# (e\Z_{\geq 0} + i) \cap \D(\lambda) - \# (e\Z_{< 0} + i)\cap \D(\lambda)^c$. We define
\[
\xrestr_i(\lambda) \coloneqq  \# (e\Z_{\geq 0} + i) \cap \D(\lambda) - \# (e\llbracket -t^2,-1\rrbracket + i)\cap \D(\lambda)^c.
\]
We have
\begin{equation}
\label{equation:xrestr-xi}
\xrestr_i(\lambda) - x_i(\lambda) = \#(e\Z_{< -t^2}+i)\cap \D(\lambda)^c,
\end{equation}
and
\begin{align}
\xrestr_i(\lambda)
&=
\# (e\Z_{\geq 0} + i) \cap \D(\lambda) - \Bigl[ \#\bigl( e\llbracket -t^2,-1\rrbracket+i\bigr) - \#\bigl(e\llbracket -t^2,-1\rrbracket+i\bigr) \cap \D(\lambda)\Bigr]
\notag
\\
\label{equation:xrestr}
&=
\# (e\Z_{\geq -t^2} +i)\cap \D(\lambda) - t^2.
\end{align}

\begin{lemma}
\label{lemma:etr_0}
For any $0 \leq i < e$ and $r \geq 1$ we have
\[
\E_t\bigl|\xrestr_i(\lambda) - x_i(\lambda)\bigr|^r = o\bigl(t^{-t}),
\]
as $t \to +\infty$. In particular, the variable $\xrestr_i(\lambda) - \myxi(\lambda)$ converges to $0$   in distribution under $\plt$ as $t \to +\infty$.
\end{lemma}

\begin{proof}
Recall from~\eqref{equation:D(lambda')} that $\D(\lambda)^c = -\D(\lambda')-1$, so that
\begin{align*}
\xrestr_i(\lambda) - x_i(\lambda)
&= \#(e\Z_{<-t^2}+i)\cap \D(\lambda)^c
\\
&=
\#(e\Z_{<-t^2}+i)\cap\bigl(-\D(\lambda')-1\bigr)
\\
&=
\#(e\Z_{>t^2}-i-1)\cap \D(\lambda')
\\
&=
\#(e\Z_{\geq t^2} + \iet) \cap \D(\lambda'),
\end{align*}
where $\iet = e-1-i \in \{0,\dots,e-1\}$.
Hence, by~\eqref{equation:pl(Lambda)} it suffices to prove the result for the random variable $\#(e\Z_{\geq t^2} +\iet)\cap \D(\lambda)$.

We first prove that 
\begin{equation}
\label{equation:eZi'}
\#(e\Z_{\geq t^2}+\iet) \cap \D(\lambda) \leq \max(T-t^2,0),
\end{equation}
where $T  \coloneqq \lvert \lambda\rvert$ is (by definition) a Poisson random variable with parameter $t$. If $\#\bigl(e\Z_{\geq t^2}+\iet\bigr) \cap \D(\lambda) = n$ with $n \geq 1$ then $\D(\lambda)$ contains an element of the form $ek+\iet$ for $k \geq t^2+n-1$. Hence, we have, recalling that $\iet \geq 0$, $n \geq 1$ and $e \geq 2$,
\[
\lambda_1 -1 = \max \D(\lambda) \geq e(t^2+n-1)+\iet \geq t^2+n-1.
\]
Thus, we obtain that
\[T \geq \lambda_1 \geq t^2+n,
 \]
and thus $T - t^2 \geq n$ thus $n \leq \max\bigl(T-t^2,0\bigr)$ as desired.
This inequality is satisfied when $n = 0$ as well, proving~\eqref{equation:eZi'}.

We deduce from~\eqref{equation:eZi'}  that, using Cauchy--Schwarz inequality,
\begin{align}
\Et\bigl|\#(e\Z_{\geq t^2}+\iet) \cap \D(\lambda)\bigr|^r
&\leq \Et \left[\max(T-t^2,0)^r\right]
\notag
\\
&\leq \Et\bigl[ \ind_{T \geq t^2} (T-t^2)^r\bigr]
\notag
\\
&\leq \sqrt{\Et\ind_{T\geq t^2}} \sqrt{\Et(T-t^2)^{2r}}
\notag
\\
&=
\sqrt{\mathbb{P}_t(T \geq t^2)} \sqrt{\Et(T-t^2)^{2r}}.
\label{equation:bound_Et_card}
\end{align}
The second term in the right-hand side is the square root of a polynomial $P_r(t)$ in $t$ since $\E_t T^k$ is a polynomial in $t$ for any $k \geq 0$ (we have $\E_t T^k = \sum_{a = 0}^k t^a S_{k,a}$ where $S_{k,a}$ are the Stirling numbers of the second kind, see for instance~\cite{riordan}). For the first one, we  use the following Chernoff-type bound, valid for $x > t$ (see~\cite[Theorem 5.4]{mitzenmacher}):
\[
\mathbb{P}_t(T \geq x) \leq \frac{\exp(-t)\exp(x)t^x}{x^x} = \exp\bigl(x(1+\ln t)-x\ln x-t\bigr).
\]
For $t > 1$ and $x=t^2$ we obtain
\begin{align*}
\mathbb{P}_t\bigl(T\geq t^2\bigr)
&\leq
\exp\left(t^2(1+\ln t)-2t^2\ln t-t\right)
\\
&\leq
\exp\left( -t^2\ln t +t^2-t\right).
\end{align*}
Hence, by~\eqref{equation:bound_Et_card} we deduce that
\[
t^t \Et\bigl|\#(e\Z_{\geq t^2}+\iet) \cap \D(\lambda)\bigr|^r
\leq \sqrt{P_r(t)} \exp\left(\frac{1}{2}\left(-t^2\ln t +t^2  - t\right) + t\ln t\right) \xrightarrow{t \to +\infty} 0,
\]
whence the result.
\end{proof}

We now want to apply Theorem~\ref{theorem:clt} with $\xrestr_i(\lambda)$. We thus have to compute the asymptotics of the covariances. We will make the calculations for $\myxi(\lambda)$ and then see that it gives the desired result.

\subsection{Expectation}
\label{subsection:expectation}

\begin{lemma}
\label{lemma:sum_rhot}
We have $\sum_{m =0}^{+\infty} \rhot(m) \leq t$.
\end{lemma}

\begin{proof}
Let $m \geq 0$. If $m \in \D(\lambda)$ for $\lambda \in \parts$ then $m = \lambda_a - a$ for some $a \geq 1$ thus $|\lambda| \geq \lambda_a = m + a > m$. We deduce that
\begin{align*}
\sum_{m = 0}^{+\infty} \rhot(m)
&\leq
\sum_{m = 0}^{+\infty} \plt\bigl(m \in \D(\lambda)\bigr)
\\
&\leq
\sum_{m  =0}^{+\infty} \plt\bigl(|\lambda| > m\bigr) = \Et|\lambda|,
\end{align*}
since $|\lambda| \in \Z_{\geq 0}$. We thus obtain the desired inequality since $|\lambda|$ is a Poisson variable with parameter $t$ under $\plt$.
\end{proof}

Recall that $\besselJ_n = J_n(2\sqrt t)$.

\begin{proposition}
\label{proposition:Exi}
Under $\plt$ we have,
\[
\Et x_i(\lambda) = \sum_{k \geq 0} \rhot(ek+i) - \sum_{k \geq 1} \rhot(ek - i - 1) = \sum_{s \geq i+1}\besselsquare{s} - \sum_{s = -i}^i \sum_{k=1}^{\infty} \besselsquare{ek+s}.
\]
\end{proposition}

\begin{proof}
By Lemma~\ref{lemma:x_i_card_Fr} and~\eqref{equation:corr_lambda'} we have, where $\mathbf{1}_E$ denotes the characteristic function of any set $E$,
\begin{align*}
\Et x_i(\lambda)
&=
\sum_{k \geq 0} \Et \mathbf{1}_{\D(\lambda)}(ek+i) - \sum_{k \geq 1} \Et \mathbf{1}_{\D(\lambda')}(ek-i-1)
\\
&=
\sum_{k \geq 0} \plt(ek+i \in \D(\lambda)) - \sum_{k \geq 1} \plt(ek-i-1 \in \D(\lambda'))
\\
&=
\sum_{k \geq 0} \rhot(ek+i) - \sum_{k \geq 1}\rhod^t(ek-i-1).
\end{align*}
Note that each sum above is finite by Lemma~\ref{lemma:sum_rhot}. By Theorem~\ref{theorem:BOO} we obtain
\[
\Et x_i(\lambda) = 
\sum_{k \geq 0} \kernelJ(ek+i,ek+i) - \sum_{k \geq 1} \kernelJ(ek - i -1,ek-i-1).
\]
Using Lemma~\ref{lemma:kernelJ_sum} we deduce that, noticing that we have a telescopic sum,
\begin{align*}
\Et x_i(\lambda)
&=
\sum_{k \geq 0} \sum_{s \geq 1} \besselsquare{ek+i+s} - \sum_{k \geq 1}\sum_{s \geq 1} \besselsquare{ek-i-1+s}
\\
&=
\sum_{s \geq 1} \besselsquare{i+s} + \sum_{k \geq 1} \sum_{s \geq 1}\left( \besselsquare{ek+i+s} - \besselsquare{ek-i-1+s}\right)
\\
&=
\sum_{s \geq 1} \besselsquare{i+s} - \sum_{k \geq 1} \sum_{s=-i}^i \besselsquare{ek+s}.
\end{align*}
\end{proof}

Let us now recall some facts about the Bessel functions that we will use throughout the paper. 

\begin{lemma}
For any $n \in \Z$ and $x \in \R$ we have:
\begin{subequations}
\begin{gather}
J_n(x) \in \R,
\label{equation:Jnx_R}
\\
J_{-n} = (-1)^n J_n,
\label{equation:J-n(z)}
\\
\bigl| J_n(x) \bigr| \leq 1,
\label{equation:Jn_leq_1}
\\
J_n(x)^2 = \frac{2}{\pi}\int_0^{\pi/2} J_{2n}(2x\cos \theta) d\theta,
\label{equation:Jn2}
\\
J_n(x)^2 = \frac{2}{\pi}\int_0^{\pi/2} J_0(2x\sin \theta) \cos (2n\theta) d \theta,
\label{equation:Jn2_with_J0}
\\
\sum_{m \in \Z} J_m(x) z^m = \exp\left[\frac{1}{2}x\left(z-z^{-1}\right)\right], \qquad z \in \mathbb{C}^*,
\label{equation:generating_series_bessel}
\\
2J_n' = J_{n-1} - J_{n+1}.
\label{equation:Jn_ED}
\end{gather}
Moreover, if $n \in \N$ then
\begin{equation}
\bigl| J_n(x) \bigr| \leq \frac{|x|^n}{2^n n!},
\label{equation:upper_bound_Jn}
\end{equation}
and finally
\begin{gather}
J_0^2 + 2\sum_{s = 1}^{\infty} J_s^2 = 1,
\label{equation:sum_Js_square}
\\
J_0 + 2\sum_{s = 1}^{+\infty} J_{2s} = 1.
\label{equation:sum_even_bessel}
\end{gather}
\end{subequations}
\end{lemma}

\begin{proof}
In~\cite{nist}:
for~\eqref{equation:Jnx_R} see (10.2.2),
for~\eqref{equation:J-n(z)} see (10.4.1),
for~\eqref{equation:Jn_leq_1} see (10.14.1),
for~\eqref{equation:Jn2} see (10.9.26) and~\eqref{equation:J-n(z)},
for~\eqref{equation:Jn2_with_J0} see (10.22.16) and~\eqref{equation:J-n(z)},
for~\eqref{equation:generating_series_bessel} see (10.10.1),
for~\eqref{equation:Jn_ED} see (10.6.1),
for~\eqref{equation:upper_bound_Jn} see (10.14.4),
for~\eqref{equation:sum_Js_square} see (10.23.3),
for~\eqref{equation:sum_even_bessel} see (10.12.4).
\end{proof}

We will now prove the following result.

\begin{proposition}
\label{proposition:Et_bounded}
The quantity $\Et x_i(\lambda)$ is bounded for $t \in \R$.
\end{proposition}
 
By~\eqref{equation:Jn_leq_1} and~\eqref{equation:sum_Js_square}, we already know that the first sum in the expression of $\Et x_i(\lambda)$ in Proposition~\ref{proposition:Exi}:
\[
\sum_{s \geq 1} \besselsquare{i+s},
\]
is bounded for $t \in \R$. We will now prove that the second sum:
\[
\sum_{k \geq 1} \sum_{s=-i}^i \besselsquare{ek+s},
\]
is bounded, and this will conclude the proof of Proposition~\ref{proposition:Et_bounded}. To that extent, it suffices to prove that for any $s \in \Z$ the sum
\[
A_{e,s}(x) \coloneqq \sum_{k \in \Z} J_{ek+s}^2(x),
\]
is bounded for $x \in \R$. By~\eqref{equation:Jn2}, it suffices to prove that
\[
B_{e,s}(x) \coloneqq \sum_{k \in \Z} J_{2ek+2s}(x),
\]
is bounded for $x \in \R$. Note that~\eqref{equation:upper_bound_Jn} ensures that we can permute the sum and the integral signs. The next result generalises the standard equalities~\cite[10.12.3]{nist} expressing $\cos(z\cos\theta)$ (resp. $\sin(z\cos\theta)$) as a series involving $J_{2k}(z)$ and $\cos(2k\theta)$ (resp. $J_{2k+1}(z)$ and $\sin\bigl((2k+1)\theta\bigr)$). To avoid confusion with our $i \in \Z/e\Z$, we denote by $\ic$ the complex unit.

\begin{lemma}
\label{lemma:sum_Jek+s}
For any $k \in \Z$,   $x,t \in \R$, and $\vare \in \Z_{\geq 1}$ we have, with  $\Omega_{\vare} \coloneqq \bigl\{ \frac{2\ell \pi}{\vare} : -\bigl\lceil\frac{\vare}{2}\bigr\rceil < \ell \leq \bigl\lfloor\frac{\vare}{2}\bigr\rfloor\bigr\}$,
\[
\sum_{m \in \vare\Z + k} J_m(x) \exp(\ic mt)= \frac{1}{\vare}\sum_{\omega \in \Omega_\vare} \exp \ic\bigl[-k\omega + x\sin\left(\omega+t\right)\bigr].
\]
\end{lemma}

\begin{proof}
Let $\mu_{\vare} \subseteq \mathbb{C}^*$ be the subgroup of $\vare$-th roots of unity.
For any $u \in \mathbb{C}^*$ and $\zeta \in \mu_{\vare}$, by~\eqref{equation:generating_series_bessel} we have
\[
(u\zeta)^{-k} \exp \frac{1}{2}x\bigl(u\zeta-(u\zeta)^{-1}\bigr) = \sum_{m \in \Z} J_m(x) (u\zeta)^{m-k}.
\]
Using the identity
\begin{equation}
\label{equation:sum_root_unity}
\sum_{\zeta \in \mu_{\vare}} \zeta^{m-k} = \begin{cases}
\vare, &\text{if } \vare \mid m-k,
\\
0, &\text{otherwise},
\end{cases}
\end{equation}
we obtain
\[
\sum_{\zeta \in \mu_{\vare}} (u\zeta)^{-k} \exp \frac{1}{2}x\bigl(u\zeta-(u\zeta)^{-1}\bigr)
=
 \sum_{m \in \vare\Z+k} \vare J_m(x)u^{m-k}
\]
and thus
\[
\sum_{m \in \vare\Z+k} J_m(x) u^m = \frac{1}{\vare}\sum_{\zeta \in \mu_{\vare}} \zeta^{-k}\exp \frac{1}{2}x\bigl(u\zeta-(u\zeta)^{-1}\bigr).
\]
We deduce that
\begin{align*}
\sum_{m \in \vare\Z+k} J_m(x) \exp(\ic mt)
&=
\frac{1}{\vare}\sum_{\omega \in \Omega_\vare} \exp(-\ic k\omega)\exp \frac{1}{2}x\Bigl[\exp \ic(t+\omega)-\exp\bigl(-\ic(t+\omega)\bigr)\Bigr]
\\
&=
\frac{1}{\vare}\sum_{\omega\in\Omega_\vare} \exp(-\ic k\omega)\exp \bigr[ \ic x\sin(t+\omega)\bigr],
\end{align*}
which gives the announced formula.
\end{proof}

It follows immediately from Lemma~\ref{lemma:sum_Jek+s} applied with $t = 0$ that $B_{e,s}(x)$ is bounded for $x \in \R$ and thus, by the preceding discussion, this concludes the proof of Proposition~\ref{proposition:Et_bounded}.

\subsection{Covariance}
\label{subsection:covariance}

We give here an expression for $\Covt\bigl(\myxi(\lambda),\myxi[j](\lambda)\bigr)$ for $i \neq j$.

\begin{lemma}
\label{lemma:covariance_point}
For any $y \neq x$ we have
\[
\Covt\bigl(\mathbf{1}_{\D(\lambda)}(x),\mathbf{1}_{\D(\lambda)}(y)\bigr) = - \kernelJ(x,y)^2.
\]
\end{lemma}

\begin{proof}
Using Theorem~\ref{theorem:BOO}, a direct computation gives
\begin{align*}
\Covt\left(\ind_{\D(\lambda)}(x),\ind_{\D(\lambda)}(y)\right)
&=
\Et\ind_{\D(\lambda)}(x,y) - \Et \ind_{\D(\lambda)}(x)\Et\ind_{\D(\lambda)}(y)
\\
&=
\rhot(x,y) - \rhot(x)\rhot(y)
\\
&=
\left( \rhot(x)\rhot(y) - \kernelJ(x,y)^2\right) - \rhot(x)\rhot(y)
\\
&=
- \kernelJ(x,y)^2.
\end{align*}
\end{proof}

\begin{proposition}
\label{proposition:covariance_sum_kernel}
Let $i,j \in \{0,\dots,e-1\}$ with $i \neq j$. We have, under $\plt$,
\[
\Covt\bigl(x_i(\lambda),x_j(\lambda)\bigr) = - \sum_{m \in e\Z+i} \sum_{n \in e\Z+j} \kernelJ(m,n)^2.
\]
\end{proposition}

\begin{proof}
Recall from Lemma~\ref{lemma:x_i_card_Fr} that
\begin{align*}
x_i(\lambda) &= \sum_{m \in e\Z_{\geq 0}+i} \ind_{\D(\lambda)}(m) - \sum_{m \in e\Z_{<0} +i} \bigl(1-\ind_{\D(\lambda)}(m)\bigr).
\\
&=
\sum_{m \in e\Z_{\geq 0}+i} \ind_{\D(\lambda)}(m) + \sum_{m \in e\Z_{<0} +i} \bigl(\ind_{\D(\lambda)}(m)-1\bigr).
\end{align*}
We thus obtain the desired result by Lemma~\ref{lemma:covariance_point}.
\end{proof}

%

\subsection{Asymptotics of the covariance}
\label{subsection:asymptotics_covariance}

Let $i \neq j \in \Ze$. Our aim is to compute the asymptotics of $\Covt\bigl(x_i(\lambda),x_j(\lambda)\bigr)$. This will require a considerable amount of calculations.
For convenience we define
\begin{equation}
\label{equation:kJ}
\kJ(m,n)(x) \coloneqq x \frac{J_m(x)J_{n+1}(x) - J_{m+1}(x)J_n(x)}{m-n},
\end{equation}
for $x \in \R$ and $m ,n \in \Z$ (we will in fact only use the case $m \neq n$).  Note that
\begin{equation}
\label{equation:kernelJ_kJ}
\kernelJ(m,n) = \frac12\kJ(m,n)(2\sqrt t).
\end{equation}

\subsubsection{Differentiating}

\begin{lemma}[\protect{\cite[Proposition 2.7]{boo}}]
\label{lemma:dJ}
For any $m,n\in \Z$ we have
\[
\frac{d \kJ(m,n)}{d x} = J_m J_{n+1} + J_{m+1}J_n.
\]
\end{lemma}

\begin{proof}
From Lemma~\ref{lemma:kernelJ_sum} and~\eqref{equation:kernelJ_kJ} we have $\kJ(m,n) = 2 \sum_{s \geq 1} J_{m+s}J_{n+s}$. By~\eqref{equation:Jn_ED} we obtain that
\begin{align*}
\frac{ d \kJ(m,n)}{d x} &= 2\sum_{s \geq 1} \left(J_{m+s}' J_{n+s} + J_{m+s} J'_{n+s}\right)
\\
&=
\sum_{s \geq 1} \left( \bigl[J_{m+s-1} - J_{m+s+1}\bigr]J_{n+s} + J_{m+s}\bigl[J_{n+s-1} - J_{n+s+1}\bigr]\right)
\\
&=
\sum_{s \geq 0} J_{m+s} J_{n+s+1} - \sum_{s \geq 1} J_{m+s+1} J_{n+s} + \sum_{s \geq 0} J_{m+s+1} J_{n+s} - \sum_{s \geq 1} J_{m+s} J_{n+s+1}
\\
&=
J_m J_{n+1} + J_{m+1}J_n,
\end{align*}
as desired. Note that we can both differentiate inside the sum sign and split the infinite sums thanks to~\eqref{equation:upper_bound_Jn}.
\end{proof}

In particular, combining~\eqref{equation:kJ} and Lemma~\ref{lemma:dJ} we obtain
\begin{equation}
\label{equation:dJ2}
\frac{d \kJ(m,n)^2}{d x} = 2x\frac{J_m^2 J_{n+1}^2 - J_{m+1}^2 J_n^2}{m-n}.
\end{equation}
Recalling Proposition~\ref{proposition:covariance_sum_kernel}, we are interested in computing the asymptotics of
\begin{equation}
\label{equation:def_Cij(x)}
C_{i,j}(x) \coloneqq \sum_{m \in e\Z+i} \sum_{n \in e\Z+j} \kJ(m,n)^2(x),
\end{equation}
as $x \to +\infty$.  We will see that it suffices to compute the asymptotics of $C'_{i,j}(x)$.

\begin{lemma}
\label{lemma:bound_kJ}
Let $m \neq n \in \Z$. For all $x \in \R$ we have
\[
\left|\frac{d\kJ(m,n)^2}{dx}\right| \leq \frac{|x|^{2(|m|+|n|+1)+1}}{2^{2(|m|+|n|)+1}(|m|! |n|!)^2} \left(\frac{1}{(|m|+1)^2} + \frac{1}{(|n|+1)^2}\right).
\]
\end{lemma}

\begin{proof}
By~\eqref{equation:J-n(z)} and~\eqref{equation:dJ2}, for all $x \in \R$ we have, since $m \neq n$,
\begin{align}
\left|\frac{d \kJ(m,n)^2}{dx}\right|
&=
2|x| \left|\frac{J_m(x)^2 J_{n+1}(x)^2 - J_{m+1}(x)^2 J_n(x)^2}{m-n}\right|
\notag
\\
&\leq
2|x| \Bigl(J_m(x)^2 J_{n+1}(x)^2 + J_{m+1}(x)^2 J_n(x)^2\Bigr)
\notag
\\
&\leq
2|x| \left(\frac{|x|^{2|m|}}{2^{2|m|} |m|!^2} \frac{|x|^{2(|n|+1)}}{2^{2(|n|+1)} (|n|+1)!^2} + \frac{|x|^{2(|m|+1)}}{2^{2|m|} (|m|+1)!^2}\frac{|x|^{2|n|}}{2^{2|n|}|n|!^2}\right)
\label{equation:bound_kJ2}
\\
&\leq
\frac{|x|^{2(|m|+|n|+1)+1}}{2^{2(|m|+|n|)+1}(|m|! |n|!)^2} \left(\frac{1}{(|n|+1)^2} + \frac{1}{(|m|+1)^2}\right),
\notag
\end{align}
as announced.
\end{proof}

Note that $\kJ(m,n)(0) = 0$ for all $m \neq n$ by~\eqref{equation:kJ}, thus $C_{i,j}(0) = 0$. Now for any $x > 0$ we have
\begin{align*}
\sum_{m\neq n \in \Z}\frac{x^{2|m|}}{2^{2|m|} |m|!^2} \frac{x^{2(|n|+1)}}{2^{2(|n|+1)} (|n|+1)!^2}
&\leq
\sum_{m, n \in \Z}\frac{x^{2|m|}}{2^{2|m|} |m|!^2} \frac{x^{2(|n|+1)}}{2^{2(|n|+1)} (|n|+1)!^2}
\\
&\leq
\left(\sum_{m \in \Z}\frac{x^{2|m|}}{2^{2|m|} |m|!^2}\right)\left(\sum_{n \in \Z}\frac{x^{2(|n|+1)}}{2^{2(|n|+1)} (|n|+1)!^2}\right)
\\
&\leq
\left(\sum_{m \in \Z}\frac{x^{2|m|}}{2^{2|m|} |m|!}\right)\left(\sum_{n \in \Z}\frac{x^{2(|n|+1)}}{2^{2(|n|+1)} (|n|+1)!}\right)
\\
&\leq
\left(-1 + 2\sum_{m \geq 0} \frac{x^{2m}}{2^{2m}m!}\right)\left(-\frac{x^2}{4}+ 2\sum_{n \geq 0} \frac{x^{2(n+1)}}{2^{2(n+1)}(n+1)!}\right)
\\
&\leq
\left[-1+2\exp\left(\tfrac{x^2}{4}\right)\right] \left[-\tfrac{x^2}{4} + 2\left(\exp\left(\tfrac{x^2}{4}\right)-1\right)\right],
\end{align*}
thus by~\eqref{equation:bound_kJ2}  we know that $C_{i,j}(x) < +\infty$ for all $x \in \R$ and $C_{i,j}$ is differentiable with
\[
C'_{i,j}(x) = \sum_{m \in e\Z+i} \sum_{n \in e\Z+j} \frac{d \kJ(m,n)^2}{dx},
\]
for all $x \in \R$. By~\eqref{equation:dJ2} we obtain
\begin{align*}
C'_{i,j}(x)
&=
2x \sum_{m \in e\Z+i} \sum_{n \in e\Z+j} \frac{J_m^2 J_{n+1}^2 - J_{m+1}^2 J_n^2}{m-n}
\\
&=
2x \sum_{m \in e\Z+i} \sum_{n \in e\Z + j -i} \frac{J_{m+1}^2 J_{m+n}^2 - J_m^2 J_{m+n+1}^2}{n}.
\end{align*}
Note that by~\eqref{equation:upper_bound_Jn} we can distribute the double sum with the difference (we keep the above form for the moment to avoid overloading the equalities). Using~\eqref{equation:Jn2_with_J0} we obtain
\begin{multline*}
C_{i,j}'(x)
=
\frac{4x}{\pi} \sum_{m \in e\Z+i} \sum_{n \in e\Z+j-i} J^2_{m+1}(x) \int_0^{\pi/2} J_0(2x\sin \theta) \frac{\cos 2(m+n)\theta}{n} d \theta
\\
 -J^2_m(x) \int_0^{\pi/2} J_0(2x\sin \theta) \frac{\cos 2(m+n+1)\theta}{n} d \theta.
\end{multline*}

\subsubsection{Removing the infinite sums}

We now define
\begin{multline}
\label{equation:def_D}
D(x)
\coloneqq
\frac{4x}{\pi} \sum_{m \in e\Z+i} \sum_{n \in e\Z+j-i} J^2_{m+1}(x) \int_0^{\pi/2} J_0(2x\sin \theta) \frac{\exp 2\ic(m+n)\theta}{n} d \theta
\\
 -J^2_m(x) \int_0^{\pi/2} J_0(2x\sin \theta) \frac{\exp 2\ic(m+n+1)\theta}{n} d \theta,
\end{multline}
so that
\begin{equation}
\label{equation:C'_ReD}
C'_{i,j}(x) = \Re D(x),
\end{equation}
 for all $x \in \R$, recalling from~\eqref{equation:Jnx_R} that $J_n(y) \in \R$ if $y \in \R$. We now define
\begin{equation}
\label{equation:sumexp}
\sumexp[k](x) \coloneqq \sum_{n \in e\Z + k} \frac{\exp 2\ic n x}{n} = 2\sum_{n \in 2e\Z+2k} \frac{\exp \ic nx}{n},
\end{equation}
for any $k \notin e\Z$ and $x \in \R$. The value of   $\sumexp$ is given by the Lemma~\ref{lemma:sumexp} below, whose proof is postponed until~\textsection\ref{subsection:proof_lemma}. Recall from Lemma~\ref{lemma:sum_Jek+s} that $\Omega_{2e} = \bigl\{\frac{\ell\pi}{e} : -e < \ell \leq e\bigr\} \subseteq (-\pi,\pi]$. We define $\omega_\ell \coloneqq \frac{\ell\pi}{e}$ so that $\Omega_{2e} = \{\omega_\ell\}_{-e < \ell \leq e}$. 

\begin{lemma}
\label{lemma:sumexp}
For any $k \notin e\Z$ and $x \in (-\pi,\pi) \setminus \Omega_{2e}$ we have
\[
\sumexp[k](x) = \frac{\ic}{e}\sum_{\omega \in \Omega_{2e}}\exp(2\ic k\omega)(\omega-\sgnx{x}{\omega}\pi),
\]
with $\sgnx{x}{\omega} \coloneqq \sgn(\omega-x)$. In particular, the function $\sumexp[k]$ is piecewise constant on each interval $\bigl(\omega_\ell,\omega_{\ell+1}\bigr)$ for $-e \leq \ell < e$.
\end{lemma}

From~\eqref{equation:def_D} and~\eqref{equation:sumexp} we have
\begin{multline}
\label{equation:D_before_lemma_E}
D(x)=
\frac{4x}{\pi}  \int_0^{\pi/2} J_0(2x\sin \theta) \sumexp(\theta) \sum_{m \in e\Z+i}\left[J^2_{m+1}(x) \exp(2\ic m\theta) - J^2_m(x) \exp 2\ic(m+1)\theta  \right] d \theta.
\end{multline}
Note that in~\eqref{equation:def_D} we can permute the integral and the sum in $m$ by~\eqref{equation:Jn_leq_1} and~\eqref{equation:upper_bound_Jn}. Now to permute the integral and the sum in $n$, we have to be a bit more careful. By summation by parts, we have, assuming that $k \coloneqq j-i$ with $0 \leq i < j < e$,
\begin{align}
\sum_{n \in e\Z_{> 0} + k} \frac{\exp 2\ic n\theta}{n}
&=
\sum_{n \in e\Z_{> 0}+k} \frac{\sumpartexp_n(\theta) - \sumpartexp_{n-e}(\theta)}{n}
\notag
\\
&=
-\frac{\sumpartexp_k(\theta)}{e+k} + e\sum_{n \in e\Z_{> 0} +k} \frac{\sumpartexp_n(\theta)}{n(n+e)},
\label{equation:permutation_sum}
\end{align}
where
\[
\sumpartexp_n(\theta) \coloneqq \sum_{\substack{m \in e\Z_{>0}+k \\  m \leq n}} \exp 2\ic m \theta
\]
(with a similar treatment for the remaining elements $n \in e\Z+k$).
Now writing $n = eN+k$ with $N > 0$, by a standard calculation we have, for $\theta \notin \frac{\pi}{e}\Z$,
\begin{align*}
\sumpartexp_n(\theta)
&=
\sum_{m = 1}^N \exp2\ic(em+k)\theta
\\
&=
\exp(2\ic k \theta) \sum_{m = 1}^N \exp2\ic em\theta
\\
&= \exp(2\ic k \theta) \exp\left(\ic e(N+1)\theta\right) \frac{\sin eN\theta}{\sin e\theta},
\end{align*}
thus since $\bigl|\sumpartexp_n(\theta)| \leq N \leq n$ we deduce that\footnote{The author in indebted to Arnaud Debussche for this trick.}
\[
\bigl|\sumpartexp_n(\theta)\bigr| = {\sqrt{\bigl|\sumpartexp_n(\theta)\bigr|}}^2 \leq \frac{\sqrt N}{\sqrt{|\sin e\theta|}} \leq \frac{\sqrt n}{\sqrt{|\sin e\theta|}}.
\]
We thus have, for $n \in e\Z_{> 0}+k$,
\begin{align*}
\left|\frac{\sumpartexp_n(\theta)}{n(n+e)}\right| \leq \frac{\sqrt{n}}{n(n+e)}\frac{1}{\sqrt{|\sin e\theta|}} \leq \frac{1}{(n+e)\sqrt n}\frac{1}{\sqrt{|\sin e\theta|}},
\end{align*}
thus since $\frac{1}{\sqrt {|\sin u|}}$ is integrable at $0$ we deduce that
\[
\int_0^{\pi/2} \sum_{n \in e\Z_{> 0}+k} \frac{\bigl|\sumpartexp_n(\theta)\bigr|}{n(n+e)} d\theta < +\infty,
\]
which, together with~\eqref{equation:Jn_leq_1} and recalling~\eqref{equation:permutation_sum}, justifies the permutation of the integral and the sum in $n$  between~\eqref{equation:def_D} and~\eqref{equation:D_before_lemma_E} (the reasoning for $n \in e\Z_{\leq 0} +k$ being similar).

\medskip
We now go back to~\eqref{equation:D_before_lemma_E}, which we repeat here:
\[
D(x)=
\frac{4x}{\pi}  \int_0^{\pi/2} J_0(2x\sin \theta) \sumexp(\theta) \sum_{m \in e\Z+i}\left[J^2_{m+1}(x) \exp(2\ic m\theta) - J^2_m(x) \exp 2\ic(m+1)\theta  \right] d \theta.
\]
 By~\eqref{equation:Jn2} we have
\begin{multline}
\label{equation:D}
D(x)=
\frac{8x}{\pi^2}  \iint_{[0,\pi/2]^2} J_0(2x\sin \theta) \sumexp(\theta) \sum_{m \in e\Z+i}
J_{2(m+1)}(2x\cos\eta) \exp(2\ic m\theta) 
\\- J_{2m}(2x\cos\eta) \exp 2\ic(m+1)\theta  d \theta d \eta,
\end{multline}
the permutation between the sum and the integral being justified by~\eqref{equation:upper_bound_Jn}. By Lemma~\ref{lemma:sum_Jek+s} we have
\[
\sum_{m \in 2e\Z + 2k} J_m(x) \exp(\ic mu)= \frac{1}{2e}\sum_{\omega \in \Omega_{2e}} \exp \ic\bigl[-2k\omega + x\sin\left(\omega+u\right)\bigr],
\]
for any $k \in \Z$ and $u \in \R$, thus the sum in~\eqref{equation:D} becomes, where $y$ stands for $2x\cos\eta$,
\begin{align*}
\begin{multlined}[b]
\sum_{m \in e\Z+i}
J_{2(m+1)}(y) \exp(2\ic m\theta) 
\\- J_{2m}(y) \exp 2\ic(m+1)\theta
\end{multlined}
&=
\sum_{m \in e\Z+i+1} J_{2m}(y) \exp 2\ic(m-1)\theta
-
\exp 2\ic\theta\sum_{m \in e\Z+i} J_{2m}(y) \exp 2\ic m\theta
\\
&=
\begin{multlined}[t]
\exp(-2\ic\theta)\sum_{m \in 2e\Z+2(i+1)} J_m(y) \exp \ic m\theta
\\
-
\exp 2\ic\theta \sum_{m \in 2e\Z+2i} J_{m}(y) \exp \ic m\theta
\end{multlined}
\\
&=
\begin{multlined}[t]
\frac{1}{2e}\sum_{\omega \in \Omega_{2e}}
\exp(-2\ic\theta)\exp \ic\bigl[-2(i+1)\omega + y\sin(\omega+\theta)\bigr]
\\
-
\exp (2\ic\theta) \exp \ic\bigl[-2i\omega + y\sin(\omega+\theta)\bigr]
\end{multlined}
\\
&=
\frac{1}{2e}\sum_{\omega \in \Omega_{2e}} \bigl[\exp \bigl(-2\ic(\theta+\omega)\bigr) - \exp(2\ic\theta)  \bigr] \exp \ic\bigl[-2i\omega + y \sin(\theta+\omega)\bigr].
\end{align*}
Defining, for any $u,\theta \in \R$ and $\omega \in \Omega_{2e}$, the quantities
\begin{align}
\besselexp(u) &\coloneqq \frac{2}{\pi}\int_0^{\pi/2} \exp \ic\bigl[u\cos\eta\bigr] d\eta,
\label{equation:besselexp}
\\
\label{equation:fell}
f_\omega(\theta) &\coloneqq   \sumexp(\theta)\Bigl[\exp\bigl(-2\ic(\theta+\omega)\bigr) 
-  \exp (2\ic\theta) \Bigr]\exp (-2\ic i \omega),
\end{align}
we thus obtain, recalling~\eqref{equation:D},
\[
D(x)
=
\frac{2x}{e\pi}  \sum_{\omega \in \Omega_{2e}} \int_0^{\pi/2} J_0(2x\sin \theta)f_\omega(\theta) \besselexp\bigl(2x\sin(\theta+\omega)\bigr) d\theta.
\]
We now fix $\omega \in \Omega_{2e}$ and we study the asymptotics as $x \to +\infty$ of
\begin{equation}
\label{equation:Dell}
D_\omega(x)
=
\frac{2x}{e\pi}  \int_0^{\pi/2} J_0(2x\sin \theta)f_\omega(\theta) \besselexp\bigl(2x\sin(\theta+\omega)\bigr) d\theta,
\end{equation}
so that
\begin{equation}
\label{equation:D_sum_Dell}
D = \sum_{\omega \in \Omega_{2e}} D_\omega.
\end{equation}

\subsubsection{Negligible terms}

By~\cite[10.9.1]{nist}, for $y \in \R$ we have
\[
J_0(y) = \frac{1}{\pi}\int_0^\pi \cos\bigl(y\cos(\eta)\bigr)d\eta,
\]
thus $J_0(y) = \frac{2}{\pi}\int_0^{\pi/2} \cos\bigl(y\cos(\eta)\bigr)d\eta = \Re\besselexp(y)$. By \cite[10.17.3]{nist} we have, for $x > 0$,
\begin{equation}
\label{equation:bessel0_asymptotics}
J_0(x) = \sqrt{\frac{2}{\pi x}}\left(\cos\left(x - \frac{\pi}{4}\right) + M(x)\right), \qquad M(x) =_{x \to +\infty} O\left(\frac{1}{x}\right).
\end{equation}
Since $M(x) = \sqrt{\frac{\pi x}{2}}J_0(x) - \cos\left(x - \frac{\pi}{4}\right)$ for all $x \geq 0$, the function $M$ is continuous on $\R_+$ and thus there exists a constant $C \geq 0$ such that
\[\lvert M(x)\rvert \leq \frac{C}{1+x},
\]
 for all $x \geq 0$.

Now,  $\Im\besselexp(y) = \frac{2}{\pi}\int_0^{\pi/2} \sin\bigl(y\cos\eta\bigr)d\eta$ is the \emph{Struve function of zero order}, thus by   \cite[11.2.5, 11.6.1, 10.17.4]{nist} and~\eqref{equation:bessel0_asymptotics} we have, for $x >0$,
\[
\besselexp(x) = \sqrt{\frac{2}{\pi x}}\left(\exp \ic\left(x - \frac{\pi}{4}\right) + N(x)\right), \qquad N(x) =_{x \to +\infty} O\left(\frac{1}{x}\right).
\]
Using the equality $\besselexp(-x) = \overline{\besselexp(x)}$, we deduce that we have, for $x < 0$,
\[
\besselexp(x) = \sqrt{\frac{2}{-\pi x}}\left(\exp \ic\left(x + \frac{\pi}{4}\right) + \overline{N(-x)}\right), \qquad \overline{N(x)} =_{x \to +\infty} O\left(\frac{1}{x}\right).
\]
so that for all $x \neq 0$,
\begin{equation}
\label{equation:besselexp_asymptotics}
\besselexp(x) = \sqrt{\frac{2}{\pi \lvert x\rvert}}\left(\exp \ic\left(x - \sgn(x)\frac{\pi}{4}\right) + N(x)\right), \qquad N(x) =_{\lvert x\rvert \to +\infty} O\left(\frac{1}{x}\right).
\end{equation}
Similarly to $M$, we can take $C$ large enough so that for any $x \in \R$ we have 
\begin{equation}
\label{equation:ineq_N}
\lvert N(x)\rvert \leq \frac{C}{1+\lvert x\rvert}.
\end{equation}

Let us now recall the following version of the Riemann--Lebesgue lemma (we provide a proof in~\textsection\ref{subsection:riemann-lebesgue}).

\begin{lemma}[Riemann--Lebesgue]
\label{lemma:riemann-lebesgue}
Let $f$ be integrable on $(a,b)$ and $\phi$ be continuously differentiable on $[a,b]$. If $\phi$ vanishes at a finite number of points then
\[
\int_a^b f(t) \exp \ic\bigl(x\phi(t)\bigr) dt \xrightarrow{x\to+\infty} 0.
\]
\end{lemma}

\begin{proposition}
\label{proposition:Dell_O(1)}
If $\omega \notin \{0,\pi\}$ then
\[
\lim_{x \to +\infty} D_\omega(x) = 0.
\]
\end{proposition}

\begin{proof}
Recalling~\eqref{equation:Dell}, write
\[
D_\omega(x) = \frac{1}{e\pi}  \int_0^{\pi/2} \sqrt{2x\sin\theta} J_0(2x\sin \theta) \sqrt{2x\lvert\sin(\theta+\omega)\rvert} \besselexp\bigl(2x\sin(\theta+\omega)\bigr) 
g_\omega(\theta) d\theta,
\]
where $g_\omega(\theta) \coloneqq \frac{f_\omega(\theta) }{\sqrt{\sin(\theta)\lvert \sin(\theta+\omega)\rvert}}$.  Note that $f_\omega$ is piecewise continuous, moreover by assumption if $\sin(\theta) = 0$ then $\sin(\theta+\omega) \neq 0$ (since $\omega \not\equiv 0 \pmod \pi$) thus $\theta \mapsto \frac{1}{\sqrt{\sin(\theta)\lvert\sin(\theta+\omega)\rvert}}$ and thus $g_\omega$ is integrable on $(0,\pi/2)$.

Using~\eqref{equation:bessel0_asymptotics} and~\eqref{equation:besselexp_asymptotics} we obtain for $D_\omega(x)$ a sum of these following four integrals (with some multiplicative coefficients that do not depend on $x$):
\begin{gather*}
\int_0^{\pi/2} \cos\left(2x\sin\theta-\frac{\pi}{4}\right) \exp \ic\left(2x\sin(\theta+\omega) - \sgn\bigl(\sin(\theta+\omega)\bigr)\frac{\pi}{4}\right)  g_\omega(\theta)d\theta,
\\
\int_0^{\pi/2} \cos\left(2x\sin\theta-\frac{\pi}{4}\right) N\bigl(2x\sin(\theta+\omega)\bigr) g_\omega(\theta)d\theta,
\\
\int_0^{\pi/2} \exp \ic\left(2x\sin(\theta+\omega) - \sgn\bigl(\sin(\theta+\omega)\bigr)\frac{\pi}{4}\right) M(2x\sin\theta) g_\omega(\theta)d\theta,
\\
\int_0^{\pi/2} M(2x\sin\theta)N\bigl(2x\sin(\theta+\omega)\bigr)g_\omega(\theta)d\theta.
\end{gather*}
We can easily see that the last three integrals go to zero as $x$ grows to infinity. For instance, we have, using~\eqref{equation:ineq_N},
\begin{multline*}
\left\lvert\int_0^{\pi/2} \cos\left(2x\sin\theta-\frac{\pi}{4}\right) N\bigl(2x\sin(\theta+\omega)\bigr) g_\omega(\theta)d\theta\right\rvert
\leq
C \int_0^{\pi/2} \frac{\lvert g_\omega(\theta)\rvert}{1+2x\bigl\lvert\sin(\theta+\omega)\bigr\rvert} d\theta,
\end{multline*}
and since
\[
\frac{\lvert g_\omega(\theta)\rvert}{1+2x\bigl\lvert\sin(\theta+\omega)\bigr\rvert} \leq \bigl\lvert g_\omega(\theta)\rvert,
\]
is integrable on $(0,\pi/2)$ and
\[
\frac{\lvert g_\omega(\theta)\rvert}{1+2x\bigl\lvert\sin(\theta+\omega)\bigr\rvert} \xrightarrow{x\to+\infty} 0,
\]
for almost all $\theta \in (0,\pi/2)$ (all except maybe the only $\theta$ for which $\sin(\theta+\omega) = 0$)
we obtain the result by dominated convergence.

For the first integral, we have
\begin{align}
\int_0^{\pi/2}g_\omega(\theta) \cos&\left(2x\sin\theta-\frac{\pi}{4}\right) \exp \ic\left(2x\sin(\theta+\omega) - \sgn\bigl(\sin(\theta+\omega)\bigr)\frac{\pi}{4}\right)  d\theta
\notag
\\
&=
\frac{1}{2}
\sum_{\epsilon \in \{\pm 1\}}\int_0^{\pi/2} g_\omega(\theta) \exp \epsilon \ic\left(2x\sin\theta-\frac{\pi}{4}\right) \exp \ic\left(2x\sin(\theta+\omega) - \sgn\bigl(\sin(\theta+\omega)\bigr)\frac{\pi}{4}\right) d\theta
\notag
\\
\label{equation:application_RL}
&=
\frac{1}{2}
\sum_{\epsilon \in \{\pm 1\}}\int_0^{\pi/2} h_{\omega,\epsilon}(\theta)\exp \bigl(2\ic x \phi_\epsilon(t)\bigr) d\theta,
\end{align}
with
\[
\phi_\epsilon(t) = \epsilon \sin\theta + \sin(\theta+\omega),
\]
and
\[
h_{\omega,\epsilon}(\theta) = g_\omega(\theta)\exp \left[ \frac{-\ic\pi}{4}\bigl(\epsilon + \sgn\bigl(\sin(\theta+\omega)\bigr)\bigr)\right].
\]
The function $h_{\omega,\epsilon}$ is integrable on $(0,\pi/2)$, and we will now see that $\phi'_\epsilon$ vanishes at at most one point. We have $\phi'_\epsilon(\theta) = \epsilon \cos\theta + \cos(\theta+\omega)$. For $\epsilon = 1$ we have
\begin{align*}
\phi'_1(\theta) = 0
&\iff
\cos(\theta+\omega) = -\cos\theta
\\
&\iff
\theta+\omega \equiv \pi \pm \theta \pmod{2\pi},
\end{align*}
thus since $\omega \neq \pi$  we obtain
\begin{align*}
\phi'_1(\theta) = 0 &\iff \theta + \omega  \equiv \pi -\theta \pmod{2\pi}
\\
&\iff \theta \equiv \frac{\pi-\omega}{2} \pmod{2\pi}.
\end{align*}
Similarly, for $\epsilon = -1$ we have
\begin{align*}
\phi'_{-1}(\theta) = 0
&\iff
\cos(\theta+\omega) = \cos\theta
\\
&\iff
\theta+\omega \equiv \pm \theta \pmod{2\pi}
\end{align*}
thus since $\omega \neq 0$  we obtain
\begin{align*}
\phi'_{-1}(\theta) = 0 &\iff \theta + \omega \equiv -\theta \pmod{2\pi}
\\
&\iff \theta = \pi - \frac{\omega}{2}.
\end{align*}
We thus conclude, applying the Riemann--Lebesgue lemma (Lemma~\ref{lemma:riemann-lebesgue}), that the quantity in~\eqref{equation:application_RL} goes to zero as $x \to +\infty$ and this finishes the proof of the proposition.
\end{proof}

\subsubsection{Contributing terms}

If $\omega \in \{0,\pi\}$ then $2\omega \equiv 0 \pmod {2\pi}$ and thus, recalling~\eqref{equation:fell},
\begin{equation}
\label{equation:f_ell_theta0pi}
f_\omega(\theta) = -2\ic\sin(2\theta) \sumexp(\theta).
\end{equation}
We deduce from~\eqref{equation:Dell} and~\eqref{equation:f_ell_theta0pi} that if $\omega \in \{0,\pi\}$ then
\[
D_\omega(x) = -\frac{4\ic x}{e\pi} \int_{0}^{\pi/2} \sumexp(\theta) J_0(2x\sin \theta)\sin(2\theta) \besselexp\bigl(2x\sin(\theta+\omega)\bigr) d\theta.
\]
In particular, noting from~\eqref{equation:besselexp} that $\besselexp(-y) = \overline{\besselexp(y)}$ we have
\begin{subequations}
\label{subequations:D0_Dpi}
\begin{align}
D_0(x) &= -\frac{4\ic x}{e\pi} \int_{0}^{\pi/2} \sumexp(\theta) J_0(2x\sin \theta)\sin(2\theta) \besselexp(2x\sin\theta) d\theta,
\label{equation:D0}
\\
D_\pi(x) &= -\frac{4\ic x}{e\pi} \int_{0}^{\pi/2} \sumexp(\theta) J_0(2x\sin \theta)\sin(2\theta) \overline{\besselexp(2x\sin\theta)} d\theta.
\label{equation:Dpi}
\end{align}
\end{subequations}
Recall that we have assumed at the beginning of~\textsection\ref{subsection:asymptotics_covariance} that $i \neq j$. Recall also the notation $\omega_\ell = \frac{\ell\pi}{e}$ for $-e \leq \ell \leq e$.
\begin{proposition}
\label{proposition:omega_0_pi}
We have
\[
\lim_{x \to +\infty} D_0(x) = \lim_{x \to +\infty} D_{\pi}(x) = \limd[j-i],
\]
where
\[
\limd \coloneqq \left(\frac{2}{e\pi}\right)^2\left(\sum_{\omega \in \Omega_{2e}} \omega \exp 2\ic k \omega + 2\pi \sum_{n=1}^{\ep-1} (1-\sin\omega_n) \exp 2\ic k\omega_n \right),
\]
for $k \notin e\Z$ and $\ep \coloneqq \bigl\lceil\frac{e}{2}\bigr\rceil$.
\end{proposition}

\begin{proof}
By~\eqref{equation:bessel0_asymptotics} and~\eqref{equation:besselexp_asymptotics}, for $t \geq 0$ we have
\begin{align}
tJ_0(t) \besselexp(t) &= \frac{2}{\pi} \cos\left(t - \frac{\pi}{4}\right)\exp \ic\left(t - \frac{\pi}{4}\right) + R(t)
\notag
\\
&=
\frac{1}{\pi} \left[2\cos^2\left(t - \frac{\pi}{4}\right) + \ic \sin \left(2t - \frac{\pi}{2}\right)\right] + R(t)
\notag
\\
&=
\frac{1}{\pi} \left[\cos\left(2t - \frac{\pi}{2}\right) + 1 + \ic\sin \left(2t - \frac{\pi}{2}\right)\right] + R(t)
\notag
\\
\label{equation:tJF}
&=
\frac{1}{\pi} + \frac{1}{\pi}\bigl(\sin(2t) - \ic\cos(2t)\bigr) + R(t),
\end{align}
where $R$ is continuous on $\R_+$ and satisfies $R(t)= O\bigl(\frac{1}{t}\bigr)$ when $t \to +\infty$. In particular, there exists $C \geq 0$ such that $\lvert R(t)\rvert \leq \frac{C}{1+t}$ for all $t \geq 0$.

Recall from Lemma~\ref{lemma:sumexp} that $\sumexp$ is constant on each interval $\bigl(\omega_\ell,\omega_{\ell+1}\bigr)$ for each $-e \leq \ell < e$. From~\eqref{equation:D0} and~\eqref{equation:tJF} we have, for any $0 \leq a < b \leq \frac{\pi}{2}$,
\begin{align}
\frac{4x}{e\pi}\int_a^b \sin(2\theta) J_0(2x\sin\theta) &\besselexp(2x\sin\theta)d\theta
\notag
\\
&=
\frac{4}{e\pi}\int_a^b  2x\sin(\theta)J_0(2x\sin\theta)\besselexp(2x\sin\theta) \cos(\theta)d\theta
\notag
\\
&=
\frac{2}{e\pi}\int_{2\sin a}^{2\sin b} xu J_0(xu) \besselexp(xu) du
\notag
\\
\label{equation:Well_transformed}
&=
\frac{2}{e\pi^2}\int_{2\sin a}^{2\sin b} 1 + \sin(2xu) - \ic\cos(2xu) + \pi R(xu) du,
\end{align}
where we did the variable change $u = 2\sin\theta$. Moreover, as $x \to +\infty$ we have
\[
\int_{2\sin a}^{2\sin b} \sin(2xu) - \ic \cos(2xu) du
= o(1),
\]
by the Riemann--Lebesgue lemma, and
\begin{align*}
\left\lvert \int_{2\sin a}^{2\sin b}  R(xu) du\right\rvert
&\leq 
\int_{2\sin a}^{2\sin b}  \frac{C}{1+xu}du
\\
&=
\frac{C}{x}\left[\ln(1+xu)\right]_{2\sin a}^{2\sin b} 
\\
&=
o(1),
\end{align*}
so that~\eqref{equation:Well_transformed} gives
\begin{equation}
\label{equation:D0_equivalent}
\frac{4x}{e\pi}\int_{a}^{b} \sin(2\theta) J_0(2x\sin\theta) \besselexp(2x\sin\theta)d\theta = \frac{4}{e\pi^2}(\sin a - \sin b) + o(1),
\end{equation}
as $x \to +\infty$.

Now define $\ep \coloneqq \bigl\lceil\frac{e}{2}\bigr\rceil$ so that $\omega_{\ep} \geq \pi/2 > \omega_{\ep-1}$.
Define $\homega_\ell$ for $\ell \in \{0,\dots,\ep\}$ by
\[
\homega_\ell \coloneqq \min\left(\omega_\ell,\frac{\pi}{2}\right) = \begin{cases}
\omega_\ell, &\text{if } \ell < \ep,
\\
\frac{\pi}{2},&\text{if } \ell = \ep.
\end{cases}
\]
We deduce from~\eqref{equation:Jnx_R}, \eqref{subequations:D0_Dpi} and~\eqref{equation:D0_equivalent} that,   denoting by $\valsumexp$ the value of $\sumexp$ on the interval $(\omega_\ell,\omega_{\ell+1})$,
\begin{equation}
\label{equation:D0_o(1)}
D_0(x)
= \frac{-4\ic}{e\pi^2} \sum_{\ell=0}^{\ep-1} \valsumexp \left(\sin\hdek[\ell+1] - \sin\hdek\right) + o(1)
= D_\pi(x),
\end{equation}
as $x \to +\infty$. Now we have, noting that $\sin\homega_{e'} = 1$ and $\sin\homega_0 = 0$,
\begin{align}
\sum_{\ell=0}^{\ep-1} \valsumexp\bigl(\sin\homega_{\ell+1}-\sin\homega_\ell\bigr)
&=
\sum_{\ell = 1}^{\ep} \valsumexp[\ell-1]\sin\homega_\ell - \sum_{\ell=0}^{\ep-1} \valsumexp\sin\homega_\ell
\notag
\\
&=
\valsumexp[\ep-1]\sin\homega_{\ep} - \valsumexp[0]\sin\homega_0 + \sum_{\ell=1}^{\ep-1} (\valsumexp[\ell-1]-\valsumexp[\ell])\sin\homega_\ell
\notag
\\
\label{equation:sum_valsumexp_diff_sin}
&=
\valsumexp[\ep-1] + 
\sum_{\ell=1}^{\ep-1} (\valsumexp[\ell-1]-\valsumexp[\ell])\sin\omega_\ell.
\end{align}
By Lemma~\ref{lemma:sumexp} we have, for any $x \in (\omega_\ell,\omega_{\ell+1})$ and with $k \coloneqq j-i \pmod e$ (thus $k \notin e\Z$),
\begin{align}
\valsumexp &= \frac{\ic}{e}\sum_{\omega \in \Omega_{2e}} \exp(2\ic k\omega)\bigl(\omega-\sgn(\omega-x)\pi\bigr)
\notag
\\
&=
\frac{\ic}{e}\sum_{\omega \in \Omega_{2e}} \omega \exp 2\ic k \omega + 
\frac{\ic\pi}{e}\sum_{n = -e+1}^{\ell} \exp 2\ic k\omega_n- \frac{\ic\pi}{e}\sum_{n = \ell+1}^e \exp 2\ic k\omega_n
\notag
\\
&=
\frac{\ic}{e}\sum_{\omega \in \Omega_{2e}} \omega \exp 2\ic k \omega - \frac{2\ic\pi}{e}\sum_{n = \ell+1}^e \exp 2\ic k\omega_n
\notag
\\
\label{equation:valsumexp_step}
&=
\frac{\ic}{e}\sum_{\omega \in \Omega_{2e}} \omega \exp 2\ic k \omega + \frac{2\ic\pi}{e}\sum_{n = 1}^\ell \exp 2\ic k\omega_n
\end{align}
since $\sum_{\omega \in \Omega_{2e}} \exp 2\ic k\omega = 0 = \sum_{n = 1}^e \exp 2\ic k\omega_n$ since $2k \notin 2e\Z$, thus
\begin{align}
\valsumexp[\ell-1] - \valsumexp
&=
\frac{2\ic\pi}{e}\sum_{n = 1}^{\ell-1} \exp 2\ic k\omega_n
- \frac{2\ic\pi}{e}\sum_{n = 1}^\ell \exp 2\ic k\omega_n
\notag
\\
\label{equation:el-1-el}
&=
\frac{-2\ic\pi}{e}\exp(2\ic k\omega_\ell).
\end{align}
Gathering~\eqref{equation:D0_o(1)}, \eqref{equation:sum_valsumexp_diff_sin}, \eqref{equation:valsumexp_step} and~\eqref{equation:el-1-el} gives the desired result.
\end{proof}

\subsubsection{Making the limit explicit}

We are now interested in a closed form for the elements $\Re \limd$, where
\[
\limd = \left(\frac{2}{e\pi}\right)^2\left(\sum_{\omega \in \Omega_{2e}} \omega \exp 2\ic k \omega + 2\pi \sum_{n=1}^{\ep-1} (1-\sin\omega_n) \exp 2\ic k\omega_n \right),
\]
are the elements of Proposition~\ref{proposition:omega_0_pi}, where $k \notin e\Z$ and $\ep = \bigl\lceil\frac{e}{2}\bigr\rceil$.

\begin{lemma}
\label{lemma:Re_limd}
For any $k \notin e\Z$ we have
\[
\Re\limd = \frac{2}{\pi e^2} \left[\cot\bigl(k-\half\bigr)\tfrac{\pi}{e} - \cot\bigl(k+\half\bigr)\tfrac{\pi}{e}\right].
\]
In particular, we have $\Re\limd \neq 0$.
\end{lemma}

\begin{proof}
First, we have
\begin{align*}
\Re \sum_{\omega \in \Omega_{2e}} \omega \exp(2\ic k\omega)
&=
\sum_{\omega\in\Omega_{2e}} \omega \cos 2 k\omega
\\
&=
\pi,
\end{align*}
noticing that if $\omega \in \Omega_{2e} \setminus\{\pi\}$ then $-\omega \in \Omega_{2e}$. We now have to deal with the sum
\[
A \coloneqq \Re\left(2\sum_{n=1}^{\ep-1} (1-\sin\omega_n) \exp 2\ic k\omega_n\right) = 2\sum_{n=1}^{\ep-1}(1-\sin\omega_n)\cos 2 k\omega_n.
\]
Recall that $\omega_n = \frac{\pi n}{e}$. For any $n \in \{1,\dots,e-1\}$ we  thus have $\omega_{e-n} = \pi-\omega_n$, hence $\sin\omega_{e-n} = \sin\omega_n$ and $\cos2 k\omega_{e-n} = \cos(2k\pi - 2k\omega_n)= \cos 2k\omega_n$. Hence, we have
\begin{align*}
A &= \sum_{n = 1}^{e-1} (1-\sin\omega_n)\cos 2k\omega_n - \begin{cases}
\bigl(1-\sin\omega_{e/2}\bigr)\cos 2k\omega_{e/2}, &\text{if } n \text{ is even},
\\
0, &\text{if } n \text{ is odd},
\end{cases}
\\
&=
\sum_{n = 1}^{e-1} (1-\sin\omega_n)\cos 2k\omega_n,
\end{align*}
since $\sin\omega_{e/2} = \sin\frac{\pi}{2} = 1$ if $e$ is even. Since $\sum_{n = 1}^{e-1} \cos 2k\omega_n = \sum_{n = 1}^{e-1} \cos \bigl(\frac{2k\pi}{e}n\bigr) = -1$ since $2k\notin 2e\Z$, we obtain that, using the identity $\sin(x)\cos(y) = \frac{1}{2}\bigl[\sin(x+y) + \sin(x-y)\bigr]$,
\begin{align*}
A &=-1 -\sum_{n = 1}^{e-1} \sin \omega_n \cos 2k\omega_n
\\
&=
-1+\frac{1}{2}\sum_{n = 1}^{e-1} \bigl[\sin(2k-1)\omega_n - \sin(2k+1)\omega_n\bigr].
\end{align*}
We now recall the following standard formula (see, for instance, \cite[1.342.1]{grad}):
\[
\sum_{n = 1}^{e-1} \sin nx = \sin \frac{ex}{2}\sin \frac{(e-1)x}{2}\cosec\frac{x}{2},
\]
valid for any $x \notin 2\pi\Z$. We obtain, using the identity $\sin\left(a \pm \frac{\pi}{2}\right) = \pm \cos a$,
\begin{align*}
\sum_{n = 1}^{e-1} \sin (2k\pm 1)\omega_n
&=
\sum_{n=1}^{e-1}\sin \frac{(2k\pm 1)\pi n}{e}
\\
&=
\sin \left[\left(k\pm\half\right)\pi\right] \sin \left[(1-e^{-1})\left(k\pm\half\right)\pi\right]\cosec\left[e^{-1}\left(k\pm \half\right)\pi\right]
\\
&=
\pm (-1)^k \sin \left[(1-e^{-1})\left(k\pm\half\right)\pi\right]\cosec\left[e^{-1}\left(k\pm \half\right)\pi\right]
\\
&=
\pm (-1)^k \sin \left[\left((1-e^{-1})k \mp \half[e^{-1}] \pm \half\right)\pi\right]\cosec\left[e^{-1}\left(k\pm \half\right)\pi\right]
\\
&=
(-1)^k \cos \left[\left((1-e^{-1})k \mp \half[e^{-1}]\right)\pi\right]\cosec\left[e^{-1}\left(k\pm \half\right)\pi\right]
\\
&=
\cos \left[\left(-e^{-1}k \mp \half[e^{-1}]\right)\pi\right]\cosec\left[e^{-1}\left(k\pm \half\right)\pi\right]
\\
&=
\cos \left[e^{-1}\left(k \pm \half\right)\pi\right]\cosec\left[e^{-1}\left(k\pm \half\right)\pi\right]
\\
&=
\cot\left[e^{-1}\left(k \pm \half\right)\pi\right].
\end{align*}
and thus
\[
A =-1+ \frac{1}{2}\left(\cot\left[e^{-1}\left(k-\half\right)\pi\right] - \cot\left[e^{-1}\left(k+\half\right)\pi\right]\right).
\]
Finally, we obtain
\begin{align*}
\Re \limd &= \left(\frac{2}{e\pi}\right)^2\bigl(\pi + \pi A\bigr)
\\
&=
\frac{4}{\pi e^2}(1+A),
\end{align*}
as desired. Note that $\Re\limd \neq 0$ since the cotangent function is bijective on $(0,\pi)$ (taking $1 \leq k \leq e-1$).
\end{proof}

 We conclude this part by putting~\eqref{equation:C'_ReD}, \eqref{equation:D_sum_Dell}, Proposition~\ref{proposition:Dell_O(1)}, Proposition~\ref{proposition:omega_0_pi} and Lemma~\ref{lemma:Re_limd} to find that
\[
\lim_{x \to +\infty} C'_{i,j}(x) = 2\Re\limd[j-i] = \frac{4}{\pi e^2}\left[\cot(j-i-\half)\tfrac{\pi}{e} - \cot(j-i+\half)\tfrac{\pi}{e}\right],
\]
in particular, the limit depends only on $j-i$. Since $\Re\limd[j-i] \neq 0$ we find that
\[
C_{i,j}(x) \sim 2\Re\limd[j-i] x,
\]
as $x \to +\infty$. By~\eqref{equation:kernelJ_kJ}, \eqref{equation:def_Cij(x)} and Proposition~\ref{proposition:covariance_sum_kernel}, we have, with $i \neq j$,
\[
\Covt\bigl(x_i(\lambda),x_j(\lambda)\bigr) = \frac{-1}{4}C_{i,j}(2\sqrt t),
\]
thus $\Covt\bigl(x_i(\lambda),x_j(\lambda)\bigr) \sim 
-\Re\limd[j-i]\sqrt t$. We thus deduce from Lemma~\ref{lemma:Re_limd} the final result of this part.

\begin{theorem}
\label{theorem:limit_covt}
If $i \neq j$ then as $t \to +\infty$ we have, under the Poissonised Plancherel measure $\plt$,
\[
\Covt\bigl(x_i(\lambda),x_j(\lambda)\bigr) \sim \frac{2\sqrt t}{\pi e^2} \left[\cot\bigl(j-i+\half\bigr)\tfrac{\pi}{e} - \cot\bigl(j-i-\half\bigr)\tfrac{\pi}{e}\right].
\]
\end{theorem}

\subsection{Asymptotics of the variance}
\label{subsection:asymptotics_variance}

Recall from~\eqref{equation:sum_xi_0} that $\sum_i x_i(\lambda) = 0$. Hence, we directly obtain that
\begin{equation}
\label{equation:var_sum_covar}
\Vart x_i(\lambda) = \Covt\bigl(x_i(\lambda),x_i(\lambda)\bigr) = - \sum_{j \neq i} \Covt\bigl(x_i(\lambda),x_j(\lambda)\bigr).
\end{equation}

\begin{corollary}
\label{corollary:asymptotics_vart}
As $t \to +\infty$ we have, under $\plt$,
\[
\Vart x_i(\lambda) \sim \frac{4\sqrt t}{\pi e^2 }\cot\frac{\pi}{2e}.
\]
\end{corollary}

\begin{proof}
By Theorem~\ref{theorem:limit_covt} and~\eqref{equation:var_sum_covar}, it suffices to prove that $\sum_{k = 1}^{e-1} \Re\limd = \frac{4}{\pi e^2}\cot\frac{\pi}{2e}$. Recalling  from Lemma~\ref{lemma:Re_limd} that
\[
\limd = \frac{2}{\pi e^2} \left[\cot\bigl(k-\half\bigr)\tfrac{\pi}{e} - \cot\bigl(k+\half\bigr)\tfrac{\pi}{e}\right],
\]
we obtain, using the fact that $\cot(e-\half)\frac{\pi}{e} = \cot(\pi - \frac{\pi}{2e}) = -\cot\frac{\pi}{2e}$,
\begin{align*}
\sum_{k = 1}^{e-1} \limd &= \frac{2}{\pi e^2}\left[\cot\frac{\pi}{2e} - \cot(e - \half)\frac{\pi}{e}\right]
\\
&=\frac{4}{\pi e}\cot\frac{\pi}{2e},
\end{align*}
as desired. Note that $\cot\frac{\pi}{2e} \neq 0$ indeed since $e \geq 2$.
\end{proof}

\begin{remark}
\label{remark:cov_i=j}
 In particular, Theorem~\ref{theorem:limit_covt} remains valid in the case $i = j$,  and the limit of $t^{-1/2} \Vart \myxi(\lambda)$ as $t \to +\infty$ does not depends on $i$.
 \end{remark}
 
With Proposition~\ref{proposition:Et_bounded} and Remark~\ref{remark:cov_i=j}, we thus have proved Theorem~\ref{theoremet:et_covt} of the introduction.

\section{Size of the core of a random partition}
\label{section:size_core}

We will now use the results of the previous sections to determine the limit law of $|\overline{\lambda}|$ under $\plt$ as $t \to +\infty$.

\subsection{Limit expectation}

Write $\myvart \coloneqq \frac{4\sqrt t}{\pi e^2 }\cot\frac{\pi}{2e}$. From~\eqref{equation:size_core} we have
\[
\frac{|\overline{\lambda}|}{\myvart} = \frac{e}{2}\sum_{i \in \Ze} \frac{x_i(\lambda)^2}{\myvart} + \sum_{i = 0}^{e-1} i \frac{x_i(\lambda)}{\myvart}.
\]
From Proposition~\ref{proposition:Et_bounded} and Corollary~\ref{corollary:asymptotics_vart}, we obtain that
\[
\lim_{t \to +\infty} \Et\frac{|\overline{\lambda}|}{\myvart} = \frac{e^2}{2},
\]
and thus the following result.

\begin{proposition}
\label{proposition:limit_Et_core}
Under the Poissonised Plancherel measure $\plt$, as $t \to +\infty$ we have
\[
\Et|\overline{\lambda}| \sim \frac{2\sqrt t}{\pi}\cot\frac{\pi}{2e}.
\]
\end{proposition}

\begin{remark}
(See~\cite{lulov-pittel,ayyer-sinha}.)
Under the uniform measure on partitions of $n$, the expectation of $|\overline{\lambda}|$ is $\frac{(e-1)\sqrt{6n}}{2\pi}$. The order of the asymptotics is the same but the constant is quite different.
\end{remark}

\subsection{Limit law for \texorpdfstring{$\myxi(\lambda)$}{xilambda}}

 Recall from~\eqref{equation:xrestr} that for $i \in \{0,\dots,e-1\}$ we have
 \begin{equation}
 \label{equation:xrestr_hash}
 \xrestr_i(\lambda) = \#_t^{(i)} - t^2,
 \end{equation}
  where $\#_t^{(i)} = \#(e\Z_{\geq -t^2} + i)\cap \D(\lambda)$. In particular, we have $\Covt \bigl(\xrestr_i(\lambda),\xrestr_j(\lambda)\bigr) = \Covt \bigl(\#_t^{(i)},\#_t^{(j)}\bigr)$. 
Now define
 \begin{equation}
 \label{equation:bij}
 b_{ij} \coloneqq \begin{dcases}
 2\cot\frac{\pi}{2e}, &\text{if } i=j,
 \\
 \cot\bigl(j-i+\half\bigr)\tfrac{\pi}{e} - \cot\bigl(j-i-\half\bigr)\tfrac{\pi}{e}, &\text{if } i \neq j,
 \end{dcases}
 \end{equation}
 for $0 \leq i,j < e$ (the formula for $i \neq j$ being in fact valid for $i= j$), set $c_{ij} \coloneqq \frac{2}{\pi e^2}b_{ij}$ and let $N$ be an $e$-dimensional centred normal random vector with covariance matrix $B = (b_{ij})$. By Theorem~\ref{theorem:limit_covt} and Remark~\ref{remark:cov_i=j}, we know that $\Covt\bigl(\myxi(\lambda),\myxi[j](\lambda)\bigr) \sim c_{ij}\sqrt{t}$ as $t \to +\infty$ (recalling that $c_{ij} > 0$). 
We have
\begin{align*}
t^{-1/2}\Covt\bigl(\#_t^{(i)},\#_t^{(j)}\bigr)
&=
t^{-1/2}\Covt\bigl(\xrestr_i(\lambda),\xrestr_j(\lambda)\bigr)
\\
&=
t^{-1/2}\Covt\Bigl(x_i(\lambda) + \bigl(\xrestr_i(\lambda)-x_i(\lambda)\bigr), x_j(\lambda) + \bigl(\xrestr_i(\lambda)-x_i(\lambda)\bigr)\Bigr)
\\
&=
t^{-1/2}\Covt\bigl(x_i(\lambda),x_j(\lambda)\bigr) + t^{-1/2}o(1).
\end{align*}
Indeed, by Lemma~\ref{lemma:etr_0} and Cauchy-Schwarz inequality we have that, as $t \to +\infty$,
\[
\Covt\bigl(\xrestr_i(\lambda)-\myxi(\lambda),\xrestr_j(\lambda)-\myxi[j](\lambda)\bigr) = o\bigl(t^{-t}\bigr)
\]
and
\[
\Covt\bigl(x_i(\lambda),\xrestr_j(\lambda)-\myxi[j](\lambda)\bigr) = o\bigl(t^{\frac{1}{4}-\frac{t}{2}}\bigr),
\]
(and the same quantity with $i$ and $j$ permuted) thus both are $o(1)$. We deduce that $\Covt\bigl(\#_t^{(i)},\#_t^{(j)}\bigr) \sim_t c_{ij} \sqrt t$ and thus
\[
\frac{\Covt\bigl(\#_t^{(i)},\#_t^{(j)}\bigr)}{\sqrt{\Vart \#_t^{(i)}} \sqrt{\Vart \#_t^{(j)}}} \xrightarrow{t \to +\infty} \frac{b_{ij}}{b_{ii}}.
\]
We can thus apply Theorem~\ref{theorem:clt} to find that the vector
\[
\left(\frac{\#_t^{(i)} - \Et \#_t^{(i)}}{t^{1/4}\sqrt{c_{ii}}}\right)_{i \in \Ze},
\]
converges in distribution to the $e$-dimensional centred normal vector with covariance matrix $\bigl(\frac{b_{ij}}{b_{ii}}\bigr)$, thus the vector
\[
e\sqrt{\frac{\pi}{2}}\left( \frac{\#_t^{(i)}-\Et\#_t^{(i)}}{t^{1/4}}\right)_{i \in \Ze},
\]
converges in distribution to $N$.
 Now by~\eqref{equation:xrestr_hash} we have $\Et\xrestr_i(\lambda) = \Et\#_t^{(i)} - t^2$ and thus
\[
e\sqrt{\frac{\pi}{2}}\left(\frac{\xrestr_i(\lambda) - \Et \xrestr_i(\lambda)}{t^{1/4}}\right)_{i \in \Ze},
\]
converges in distribution to $N$ as well. By Lemma~\ref{lemma:etr_0} and Proposition~\ref{proposition:Et_bounded}, we know that $\Et \xrestr_i(\lambda)$ is bounded as $t \to +\infty$ thus the vector
 \[
e\sqrt{\frac{\pi}{2}}\left(\frac{\xrestr_i(\lambda)}{t^{1/4}}\right)_{i \in \Ze},
\]
again converges in distribution to $N$. Finally, by Slutsky's theorem and Lemma~\ref{lemma:etr_0}, we obtain the following final result (which is Theorem~\ref{theoremet:vector} of the introduction).

\begin{theorem}
\label{theorem:convergence_vector_xi}
Under the Poissonised Plancherel measure $\plt$, as $t \to +\infty$ the random vector
 \[
e\sqrt{\frac{\pi}{2}}\left(\frac{\myxi(\lambda)}{t^{1/4}}\right)_{i \in \Ze},
\]
converges in distribution to the $e$-dimensional  centred normal vector with covariance matrix~$B=(b_{ij})_{i,j\in\Ze}$ given by $b_{ij} = \cot\bigl(j-i+\half\bigr)\tfrac{\pi}{e} - \cot\bigl(j-i-\half\bigr)\tfrac{\pi}{e}$.
\end{theorem}

 \begin{remark}
In many cases of determinantal point processes (as in~\cite{costin-lebowitz,soshnikov:gaussian,bogachev-su}), the variable with which we use the central limit theorem has a variance logarithmic in the expectation. In our case this does not hold, recalling Proposition~\ref{proposition:Et_bounded} (the first order asymptotics of the expectation vanishes  by the formulas in Lemma~\ref{lemma:x_i_card_Fr}).
 \end{remark}
 
\begin{remark}
\label{remark:length_first_row}
By Remark~\ref{remark:eabacus} and, for instance, \cite[top of p.81]{rostam}, it follows from the $e$-abacus construction that 
\[
\overline\lambda_1 = e \max_{0  \leq i < e} x_i(\lambda) - i_0,
\]
for some $0 \leq i_0 < e$ (depending on the partition $\lambda$).
Hence, we obtain from Theorem~\ref{theorem:convergence_vector_xi} that, under the Poissonised Plancherel measure $\plt$, the rescaled size of the first part of the $e$-core $t^{-1/4}\overline\lambda_1$ converges in distribution as $t \to +\infty$ to $\sqrt{\frac{2}{\pi}}\max N$, where $N$ is a centred normal distribution with the covariance matrix $B$ of Theorem~\ref{theorem:convergence_vector_xi}. We are thus reduced to study the maximum of a (correlated) Gaussian distribution, which is a well-known problem.
\end{remark}

\subsection{Limit law for \texorpdfstring{$|\overline{\lambda}|$}{the size of core of lambda}}
\label{subsection:limit_law_size_core}

Note that the covariance matrix $B=(b_{ij})$ of Theorem~\ref{theorem:convergence_vector_xi} is (symmetric, positive semi-definite and) circulant. Recall from~\eqref{equation:size_core} that
\[
\lvert \overline{\lambda}\rvert = \frac{e}{2}\sum_{i = 0}^{e-1} \myxi(\lambda)^2 + \sum_{i =0}^{e-1} i x_i(\lambda).
\]
Hence, by Theorem~\ref{theorem:convergence_vector_xi} and Slutsky's theorem we know that $t^{-1/2}|\overline{\lambda}|$ has asymptotically the same law as $t^{-1/2}\frac{e}{2}\sum_{i = 0}^{e-1} \myxi(\lambda)^2$. To explicit this law, it suffices to compute the eigenvalues of $B$. Note that since each line of $B$ sums to $0$ (by~\eqref{equation:sum_xi_0}), we already know that $0$ is an eigenvalue of $B$.


\begin{lemma}
\label{lemma:eigenvalues_covariance}
The eigenvalues of $B$ are
\[
\lambda_k = 2e\sin\frac{k\pi}{e},
\]
for $k \in \{0,\dots,e-1\}$. In particular, the only zero eigenvalue is $\lambda_0$.
\end{lemma}

\begin{proof}
For any $0 \leq j < e$, we write $b_j \coloneqq b_{0j}$ the entries of the first row of $B$.  Note that since $\cot(\pi+x) = \cot(x)$, we deduce that $b_j$ is defined for $j \in \Ze$.
Since $B$ is circulant, with the permutation matrix
\[
\permat \coloneqq \begin{pmatrix}
&1
\\
&&\ddots
\\
&&&1
\\
1
\end{pmatrix},
\]
we have
\[
B = \sum_{j = 0}^{e-1} b_j \permat^j.
\]
 Now $\permat$ has each element of $\mu_e(\mathbb{C})$ for eigenvalue, an eigenvector associated with $\zeta \in \mu_e(\mathbb{C})$ being $\bigl(\zeta^j\bigr)_{0\leq j < e}$. We deduce that the eigenvalues of $B$ are the elements, with $\zeta_k \coloneqq \exp(2\ic k\pi/e)$ for $k \in \Ze$,
 \begin{equation}
 \label{equation:lambdak_sum_bj}
 \lambda_k = \sum_{j = 0}^{e-1} b_j \zeta_k^j
 =
 \sum_{j = 0}^{e-1} b_j \cos\frac{2kj\pi}{e},
 \end{equation}
 recalling that the eigenvalues of $B$ are real since $B$ is real symmetric.
Recalling that $b_j = \cot(j+\half)\frac{\pi}{e} - \cot(j-\half)\frac{\pi}{e}$, we thus have to compute the sum
\[
\sum_{j = 0}^{e-1} \left[\cot(j+\half)\frac{\pi}{e} - \cot(j-\half)\frac{\pi}{e}\right] \cos\frac{2kj\pi}{e},
\]
for $0 \leq k < e$. First, if $k = 0$ then $\lambda_0 = 0$ since $\cot$ is $\pi$-periodic. We thus now assume $k\neq 0$.

We mimic the calculation in~\cite[\textsection 4]{berndt} of the following trigonometric sum (the case $n=1$ having been first studied by Eisenstein):
\[
\sum_{j = 1}^{e-1} \sin\frac{2\pi kj}{e}\cot^n\left(\frac{\pi j}{e}\right).
\]
Let us consider the following meromorphic function over $\mathbb{C}$:
\[
g_1 : z \mapsto \frac{\exp 2\ic\pi kz}{\exp(2\ic \pi ez)-1} - \frac{\exp(-2\ic\pi kz)}{\exp(-2\ic \pi ez)-1}.
\]
The function $g_1$ has simple poles at each $z = \frac{j}{e}$ for $j \in \Z$, with residue
\begin{align*}
\frac{\exp \frac{2\ic\pi kj}{e}}{2\ic\pi e \exp(2i\pi j)} - \frac{\exp \frac{-2\ic\pi kj}{e}}{-2\ic\pi e \exp(-2\ic\pi j)}
&=
\frac{\exp \frac{2\ic\pi kj}{e}}{2\ic\pi e} - \frac{\exp \frac{-2\ic\pi kj}{e}}{-2\ic\pi e}
\\
&=
\frac{1}{\ic\pi e}\cos\frac{2\pi kj}{e}.
\end{align*}
Now the following meromorphic function over $\mathbb{C}$:
\[
g_2 : z \mapsto  \cot \pi\left(z + \frac{1}{2e}\right) - \cot\pi\left(z - \frac{1}{2e}\right),
\]
has simple poles at each $\pm\frac{1}{2e} + \Z$. Note that since $e \geq 2$ then each pole is simple indeed. Recalling that $z \cot z \sim 1$ as $z \to 0$, we find that the residue at $\pm\frac{1}{2e} + \Z$ is $\pm\frac{ 1}{\pi}$.  We thus find that, with $g \coloneqq g_1 g_2$,
\begin{align*}
\mathrm{Res}_{\frac{j}{e}} g
&=
\mathrm{Res}_{\frac{j}{e}} (g_1) g_2\left(\frac{j}{e}\right)
\\
&=
\frac{1}{\ic\pi e} \cos\frac{2\pi k j}{e}\left[\cot \pi\left(\frac{j}{e} + \frac{1}{2e}\right) - \cot\pi\left(\frac{j}{e} - \frac{1}{2e}\right)\right]
\\
&=
\frac{b_j}{\ic\pi e}\cos\frac{2\pi kj}{e},
\end{align*}
for all $j \in \Z$, and
\begin{align*}
\mathrm{Res}_{\pm \frac{1}{2e}} g
&=
g_1\left(\frac{\pm 1}{2e}\right) \mathrm{Res}_{\pm \frac{1}{2e}} g_2
\\
&=
\pm\frac{1}{\pi}\left(\frac{\exp \frac{\pm\ic \pi k}{e}}{-2} - \frac{\exp \frac{\mp\ic \pi k}{e}}{-2}\right)
\\
&=
\mp\frac{\ic}{\pi}\sin\frac{\pm\pi k}{e}
\\
&=
-\frac{\ic}{\pi} \sin\frac{\pi k}{e}.
\end{align*}

\begin{figure}
\begin{center}
\begin{tikzpicture}
\draw[-stealth] (-3,0) -- (3,0);
\draw[-stealth] (0,-3) -- (0,3);

\draw[very thick] (0,.2) -- (0,2) node[above left]{$\ic R$}  -- (1,2) node[above right]{$1 + \ic R$} node[midway]{$<$} -- (1,.2) arc (90:260:.2) -- (1,-2) node[below right]{$1-\ic R$} -- (0,-2) node[below left]{$- \ic R$} node[midway]{$>$} -- (0,-.2) arc (260:90:.2);
\end{tikzpicture}
\end{center}
\caption{Contour $C_R$}
\label{figure:contour}
\end{figure}
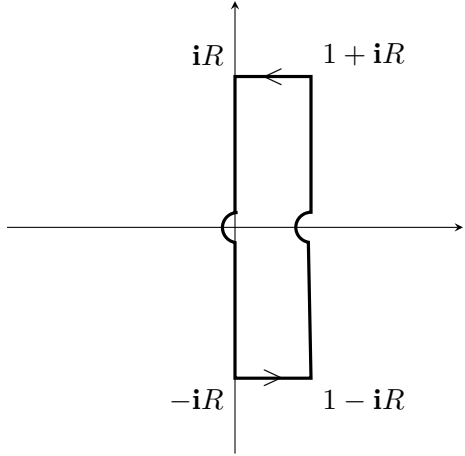

We now consider the integral of $g \coloneqq g_1 g_2$ on the contour $C_R$ given by the direct rectangle with vertices $\pm \ic R$ and $1 \pm \ic R$, with the points $0$ and $1$ avoided from the left by a small semi-circle of radius $< \frac{1}{2e}$ (see Figure~\ref{figure:contour}). Note that the only poles of $g$ inside $C_R$ are the $\frac{j}{e}$ for $0 \leq j < e$, together with $\frac{1}{2e}$ and $1-\frac{1}{2e}$.  Recalling that $\cot(\pi + z) = \cot(z)$, we obtain that $g$ is $1$-periodic thus the integral of $g$ on the two vertical sides of $C_R$ vanishes. Now for $x,y \in \mathbb{R}$ we have
\[
\frac{\exp \pm 2\ic \pi k(x+\ic y)}{\exp\bigl(\pm 2\ic \pi e(x+\ic y)\bigr)-1}
=
\frac{\exp(\pm 2\ic \pi k x )\exp(\mp 2\pi k y)}{\exp(\pm 2\ic \pi ex) \exp(\mp 2 \pi ey)-1}
\]
that goes to $0$ as $|y|\to +\infty$ (since $1 \leq k < e$)   uniformly in $x$, thus $g_1(x+\ic y) \xrightarrow{|y| \to +\infty} 0$ uniformly in $x \in \R$. Moreover, recalling from~\cite[4.21.40]{nist} that $\cot(a+\ic b) = \frac{\sin 2a - \ic \sinh 2b}{\cosh 2b - \cos 2a}$ we have:
\[
\cot\pi\left(x+\ic y \pm \frac{1}{2e}\right) 
=\frac{\sin 2\pi\left(x\pm \frac{1}{2e}\right) - \ic \sinh(2\pi y)}{-\cos 2\pi\left(x\pm \frac{1}{2e}\right) + \cosh(2\pi y)}
\]
that goes to $\pm\ic$ as $|y| \to +\infty$, uniformly in $x \in \R$.  We deduce that $g_2(x+\ic y)\xrightarrow{|y|\to+\infty} 0$ uniformly in~$x\in [0,1]$ and thus so does $g(x+\ic y)$.  Since the only poles of $g$ are on the real axis, we deduce that $\int_{C_R} g(z) dz = 0$ for any $R > 0$. Now from the residue theorem we obtain that
 \begin{align*}
 0 &= \mathrm{Res}_{\frac{1}{2e}} g + \mathrm{Res}_{1-\frac{1}{2e}} g+ \sum_{j = 0}^{e-1} \mathrm{Res}_{\frac{j}{e}} g
 \\
 &=
-\frac{2\ic}{\pi}\sin\frac{\pi k}{e} + \frac{1}{\ic \pi e}\sum_{j = 0}^{e-1} b_j \cos \frac{2\pi kj}{e},
\end{align*}
thus, recalling~\eqref{equation:lambdak_sum_bj},
\[
\lambda_k = 2e\sin\frac{\pi k}{e}.
\]
\end{proof}

Let $N = (N_i)_{0 \leq i < e}$ be a centred normal (column) vector with covariance matrix $B$ and let $Q$ be an orthogonal matrix such that $QBQ^\top = D$ with $D \coloneqq \mathrm{diag}(\lambda_j)_{0 \leq j < e}$. The random vector $(N'_i)_{0 \leq i < e}$ given by $N' = QN$ is a centred  normal vector with covariance matrix $D$ and we have $\sum_{i = 0}^{e-1} N_i^2 = \sum_{i = 0}^{e-1} N_i^{\prime 2}$.
Now $N'_i \simeq \mathcal{N}(0,\lambda_i)$ is a centred normal distribution with variance $\lambda_i$, moreover if $i \neq j$ then $N'_i$ and $N'_j$ are independent. By Lemma~\ref{lemma:eigenvalues_covariance} we have $\lambda_0 = 0 < \lambda_i$ for all $1 \leq i < e$, in particular $N_0 = 0$ almost surely and each $N_i^{\prime 2}$ for $i \geq 1$ has a Gamma distribution  $\Gamma(\half,2\lambda_i)$ with shape $\half$ and scale $2\lambda_i$. We deduce the main result of the paper (Theorem~\ref{theoremet:core} of the introduction).

\begin{theorem}
\label{theorem:size_core_plancherel}
The rescaled size $\frac{\pi}{4\sqrt t}|\overline{\lambda}|$ of the $e$-core of $\lambda$ has, under the Poissonised Plancherel measure $\plt$, asymptotically as $t \to +\infty$ the distribution of a sum of $e-1$ mutually independent $\Gamma\bigl(\half,\sin\frac{k\pi}{e}\bigr)$ for $k \in \{1,\dots,e-1\}$.
\end{theorem}

\begin{proof}
We saw at the beginning of~\textsection\ref{subsection:limit_law_size_core} that $t^{-1/2}|\overline{\lambda}|$ has asymptotically the same law as $t^{-1/2}\frac{e}{2}\sum_{i = 0}^{e-1} \myxi(\lambda)^2$. By Theorem~\ref{theorem:convergence_vector_xi} and the preceding discussion, since the vector $t^{-1/4}e\sqrt{\frac{\pi}{2}}\bigl(\myxi(\lambda)\bigr)_i$ converges in distribution to $N$, we deduce that $t^{-1/2} e^2 \frac{\pi}{2}\sum_{i = 0}^{e-1} \myxi(\lambda)^2$ converges in distribution to a random variable $G$ that is a sum of independent $\Gamma(\half,2\lambda_i)$ for $1 \leq i < e$. We deduce that $t^{-1/2}|\overline{\lambda}|$ converges in distribution to $\frac{1}{e\pi}G$. Recalling from Lemma~\ref{lemma:eigenvalues_covariance} that $\lambda_i = 2e\sin\frac{\pi i}{e}$, we have that $\frac{1}{4e}G$ is a sum of independent $\Gamma\bigl(\frac{1}{2},\sin\frac{\pi i}{e}\bigr)$. This concludes the proof since $\frac{\pi}{4\sqrt t}|\overline{\lambda}|$ converges in distribution to $\frac{1}{4e}G$.
\end{proof}

\begin{remark}
 If the Plancherel measure is replaced by the uniform one, then by~\cite{lulov-pittel,ayyer-sinha} the random variable $\frac{\pi}{\sqrt n}|\overline{\lambda}|$ converges in distribution to $\Gamma(\frac{e-1}{2},\sqrt 6)$, which is a sum of $e-1$ independent $\Gamma\bigl(\frac{1}{2},\sqrt 6\bigr)$. This situation thus ``corresponds'' to $\frac{2\lambda_k}{e} = \sqrt 6$ for all $1 \leq k < e$ (of course this is not clear at all whether $t^{-1/4} x(\lambda)$ converges in distribution to a normal vector in the uniform case).
 \end{remark}
 
 \begin{remark}
Noting that $\lambda_k = \lambda_{-k}$ for $k \in \Ze$, we obtain that the sum of $e-1$ mutually independent Gamma distributions in Theorem~\ref{theorem:size_core_plancherel} is in fact a sum of $\bigl\lfloor\frac{e}{2}\bigr\rfloor$ mutually independent Gamma distributions.
 \end{remark}

We illustrate with Figures~\ref{figure:e4} and~\ref{figure:e7} the pointwise convergence of the cumulative distribution functions in Theorem~\ref{theorem:size_core_plancherel} for the (non-Poissonised) Plancherel measures $\pl$ for $n = 100, 500, 3000$ respectively. Each (renormalised) histogram is constructed from $7000$ trials, the range $[0,5]$ being divided into $200$ bins. These simulations  indicate that Theorem~\ref{theorem:size_core_plancherel} should still hold in this non-Poissonised setting.

\begin{figure}
\begin{tabular}{ccc}
\includegraphics[scale=.3]{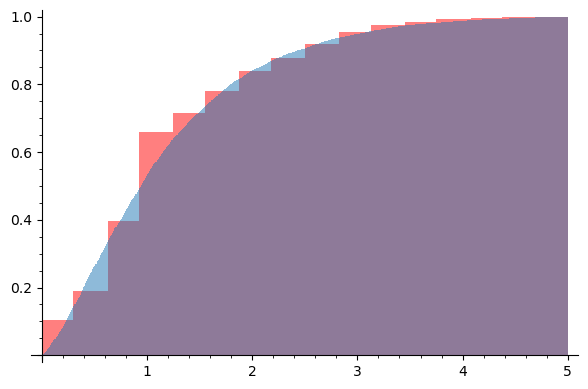}
&
\includegraphics[scale=.3]{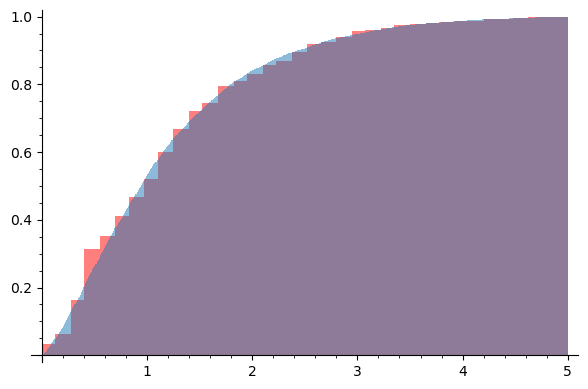}
&
\includegraphics[scale=.3]{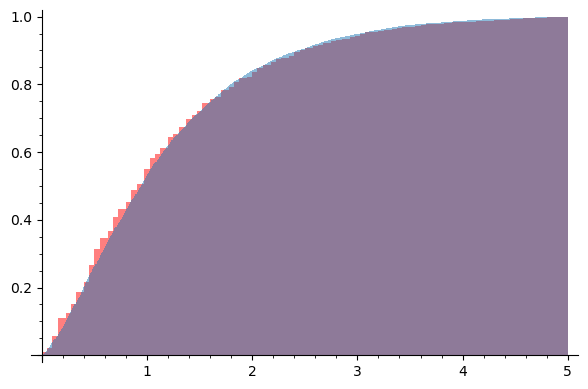}
\end{tabular}
\caption{Case $e  =4$. Convergence to $\Gamma\bigl(1,\frac{1}{\sqrt 2}\bigr) + \Gamma(\frac{1}{2},1)$.}
\label{figure:e4}
\bigskip
\begin{tabular}{ccc}
\includegraphics[scale=.3]{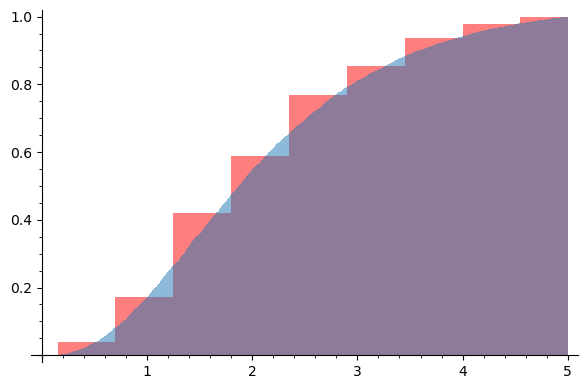}
&
\includegraphics[scale=.3]{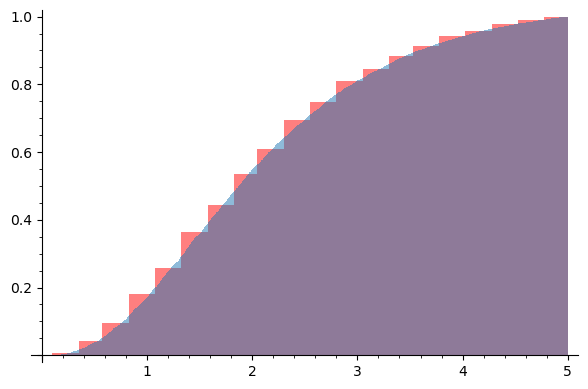}
&
\includegraphics[scale=.3]{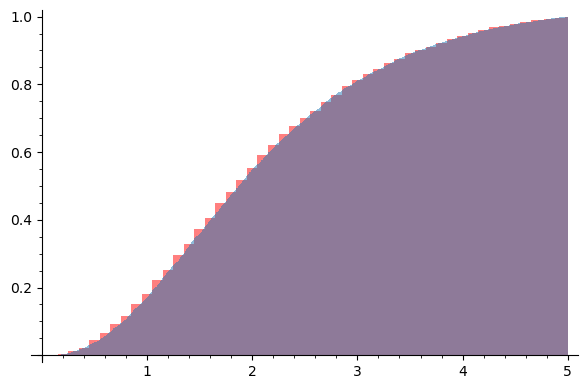}
\end{tabular}
\caption{Case $e  =7$. Convergence to $\sum_{k = 1}^3 \Gamma\bigl(1,\sin \frac{k\pi}{7}\bigr)$.}
\label{figure:e7}
 \end{figure}
 
\subsection{Limit variance}

The proof of Theorem~\ref{theorem:clt} in~\cite{soshnikov:gaussian} shows that the convergence in Theorem~\ref{theorem:clt} holds in fact in moments. In particular, the convergence of Theorem~\ref{theorem:size_core_plancherel} holds in moments and we recover Proposition~\ref{proposition:limit_Et_core}, recalling that $\E\Gamma(k,\theta) = k\theta$ and $\sum_{k = 1}^{e-1} \sin\frac{k\pi}{e} = \cot\frac{\pi}{2e}$ (see, for instance,~\cite[1.344.1]{grad}). We also obtain the following result.

\begin{corollary}
\label{corollary:limit_variance}
Under the Poissonised Plancherel measure $\plt$, as $t \to +\infty$ we have
\[
\Vart |\overline{\lambda}| \sim \frac{4et}{\pi^2}.
\]
\end{corollary}

\begin{proof}
Recall that $\Var\,\Gamma(k,\theta) = k\theta^2$, so that by Theorem~\ref{theorem:size_core_plancherel} and the preceding discussion we have
\[
\lim_{t \to +\infty} \frac{\pi^2}{16 t} \Vart |\overline{\lambda}| = \frac{1}{2} \sum_{k = 1}^{e-1} \sin^2\frac{k\pi}{e}.
\]
Now recalling the formula
\[
\sum_{k = 1}^{e-1} \sin^2 kx = \frac{e-1}{2} - \frac{\cos ex \sin(e-1)x}{2\sin x},
\]
for $x \notin \pi\Z$ (see, for instance,~\cite[1.351.1]{grad}), we obtain
\begin{align*}
\sum_{k = 1}^{e-1} \sin^2\frac{k\pi}{e}
&=
\frac{e-1}{2} - \frac{\cos \pi \sin (e-1)\frac{\pi}{e}}{2\sin\frac{\pi}{e}}
\\
&=
\frac{e-1}{2} + \frac{1}{2}
\\
&=
\frac{e}{2},
\end{align*}
and thus the desired result.
\end{proof}

\section{Number of \texorpdfstring{$i$}{i}-nodes}
\label{section:inodes}

For $i \in \Ze$, we now study the number $\ci(\lambda)$ of $i$-nodes of $\lambda$.

\subsection{Poissonised version}

Recall from~\eqref{equation:eweight} that the $e$-weight of $\lambda$ is the integer $\mathrm{w}_e(\lambda)$ such that
\[
|\lambda| = |\overline{\lambda}| + e\mathrm{w}_e(\lambda),
\]
more precisely by~\eqref{equation:cilambda_cicore} we have
\[
\ci(\lambda) = \ci(\overline{\lambda}) + \mathrm{w}_e(\lambda),
\]
for all $i \in \Ze$.  If $\lambda$ is chosen under $\plt$ and $t \to +\infty$, we know that $|\lambda|$ follows a Poisson distribution with parameter $t$ thus by the law of large numbers we have that $\frac{|\lambda|}{t}$ converges in distribution to $1$ (recalling that $|\lambda|$ is then a sum of $t$ independent Poisson variables of parameter $1$). By Theorem~\ref{theorem:size_core_plancherel}, we know that $\frac{|\overline{\lambda}|}{t}$ converges in distribution to $0$, thus by Slutsky's theorem we deduce that $\frac{\mathrm{w}_e(\lambda)}{t}$ converges in distribution to $\frac{1}{e}$.
 
 By~\eqref{equation:sum_ci} and Theorem~\ref{theorem:size_core_plancherel}, we know that $\frac{\ci(\overline{\lambda})}{t}$ converges to $0$ in distribution. We have just seen that $\frac{\mathrm{w}_e(\lambda)}{t}$ converges to $\frac{1}{e}$ in distribution, thus we have proven the following result.

\begin{proposition}
\label{proposition:ci/t}
Let $i \in \Ze$. Under the Poissonised Plancherel measure $\plt$, the random variable $\frac{\ci(\lambda)}{t}$ converges in distribution (and thus in probability) to $\frac{1}{e}$ as $t \to +\infty$.
\end{proposition} 

We will now prove the de-Poissonised version of the above result, not by using a de-Poissonisation technique but instead by using the fact that $\omega_\lambda$ converges to a limit shape.

\subsection{Law of large numbers}

Let $n \geq 1$, $\lambda \in \partsn$ and $i \in \Ze$. Recall from~\textsection\ref{subsection:young} the definition of the function $\omega_\lambda$. We denote by $\tomega_\lambda : \R \to \R$ the function defined by $\tomega_\lambda(s) \coloneqq \frac{1}{\sqrt n} \omega_\lambda\bigl(s\sqrt n\bigr)$. Note that the area between the graphs of $\tomega_\lambda$ and $|\cdot|$ is $2$. It follows from Lemma~\ref{lemma:ci_omegalambda} that
\begin{equation}
\label{equation:ci_tomegalambda}
\ci(\lambda) = \frac{\sqrt n}{2}\sum_{k \in \Z} \tomega_\lambda\left(\frac{i+ke}{\sqrt n}\right) - \left\lvert \frac{i+ke}{\sqrt n}\right\rvert.
\end{equation}

\begin{theorem}[\protect{\cite{logan-shepp:variational}, \cite{kerov-vershik:asymptotics}, \cite[Theorem 1.26]{romik}}]
\label{theorem:LLN_partitions}
Let $\Omega : \R \to \R$ be defined by
\[
\Omega(s) \coloneqq \begin{cases}
\frac{2}{\pi}\left(s\arcsin(\frac{s}{2})+\sqrt{4-s^2}\right),&\text{if } \lvert s \rvert \leq 2,
\\
\lvert s \rvert, &\text{otherwise}.
\end{cases}
\]
Then, under the Plancherel measure $\pl$, the function $\tomega_\lambda$ converges uniformly in probability to $\Omega$ as $n \to +\infty$. In other words, for any $\myepsilon > 0$ we have
\[
\pl\left(\sup_{\R} \bigl|\tomega_\lambda-\Omega\bigr| > \myepsilon\right) \xrightarrow{n\to+\infty} 0.
\]
Moreover, we also have convergence of the supports, that is:
\[
\inf\bigl\{s \in \R : \tomega_\lambda(s) \neq |s|\bigr\} \longrightarrow -2,
\]
and
\[
\sup\bigl\{s \in \R : \tomega_\lambda(s) \neq |s|\bigr\} \longrightarrow 2,
\]
in probability under $\pl$ as $n \to +\infty$.
\end{theorem}

We can now directly deduce the de-Poissonised version of Proposition~\ref{proposition:ci/t}.

\begin{proposition}
\label{proposition:ci/n}
For any $i \in \Ze$, the random variable $\frac{\ci(\lambda)}{n}$ converges in probability to $\frac{1}{e}$ under the Plancherel measure $\pl$ as $n \to +\infty$.
\end{proposition}

\begin{proof}
Take $i  \in \{0,\dots,e-1\}$ and write $\frac{\ci(\lambda)}{n} = S^i_n + R_i(\lambda)$, where
\begin{align*}
S_n^i &\coloneqq \frac{1}{2\sqrt n}\sum_{k \in \Z} \Omega\left(\frac{i+ke}{\sqrt n}\right) - \left\lvert\frac{i+ke}{\sqrt n}\right\rvert,
\\
R_i(\lambda) &\coloneqq
\frac{1}{2\sqrt n}\sum_{k \in \Z} \tomega_\lambda\left(\frac{i+ke}{\sqrt n}\right) - \Omega\left(\frac{i+ke}{\sqrt n}\right)
\end{align*}
(note that both sums are in fact finite). We have
\[
S_n^i = \frac{1}{2e}\frac{e}{\sqrt n}\sum_{k \in \Z} g\left(\frac{i+ke}{\sqrt n}\right),
\]
where $g \coloneqq \Omega - |\cdot|$ is continuous on $\R$. Since $g$ has compact support, we deduce that $S_n^i$ is  a mere Riemann sum thus $(S_n^i)_n$ converges as $n \to +\infty$ and
\[
\lim_{n \to +\infty} S_n^i = \frac{1}{2e} \int_\R g(t) dt = \frac{1}{2e}2 = \frac{1}{e}
\]
(to obtain that $\int_{\R} g(t) dt = 2$ we can just make the explicit calculation, or note from Theorem~\ref{theorem:LLN_partitions} that $g$ is a uniform limit of continuous functions of integral $2$).

Thus, it now suffices to prove that $R_i(\lambda)$ converges to $0$ in probability.
Let $\myepsilon > 0$. By Theorem~\ref{theorem:LLN_partitions}, the probability for $\lambda \in \partsn$ to satisfy
\begin{subequations}
\label{subequations:lambda_conv_Omega}
\begin{gather}
\lVert \tomega_\lambda - \Omega\rVert_\infty \leq \myepsilon,
\\
\mathrm{supp}(\tomega_\lambda - \lvert \cdot \rvert) \subseteq [-2-\myepsilon,2+\myepsilon],
\end{gather}
\end{subequations}
goes to $1$ as $n \to \infty$. Now if $|k| \geq \frac{(2+\myepsilon)\sqrt n}{e} + 1$ then 
\begin{align*}
\left|\frac{i+ke}{\sqrt n}\right|
&=
\frac{e}{\sqrt n}\left|k + \frac{i}{e}\right|
\\
&\geq \frac{e}{\sqrt n}\left(|k| - \left|\frac{i}{e}\right|\right)
\\
&\geq \frac{e}{\sqrt n}\bigl(|k| - 1\bigr)
\\
&\geq 2+\myepsilon,
\end{align*}
thus for such a $k\in\N$ we have $\bigl(\tomega_\lambda-|\cdot|\bigr)\bigl(\frac{i+ke}{\sqrt n}\bigr) = 0$.
 We deduce that for $\lambda$ satisfying~\eqref{subequations:lambda_conv_Omega} we have
\begin{align*}
\bigl\lvert R_i(\lambda)\bigr\rvert
&\leq
\frac{1}{2\sqrt n}\sum_{k \in \Z} \left\lvert \tomega_\lambda\left(\frac{i+ke}{\sqrt n}\right) - \Omega\left(\frac{i+ke}{\sqrt n}\right)\right\rvert
\\
&=
\frac{1}{2\sqrt n}\sum_{\lvert k \rvert < \frac{(2+\myepsilon)\sqrt n}{e}+1}\left\lvert\tomega_\lambda\left(\frac{i+ke}{\sqrt n}\right) - \Omega\left(\frac{i+ke}{\sqrt n}\right)\right\rvert
\\
&\leq
\frac{1}{2\sqrt n} \left(\frac{2(2+\myepsilon)\sqrt n}{e}+3\right)\myepsilon
\\
&\leq
\frac{(2+\myepsilon)\myepsilon}{e} + \frac{3\delta}{2\sqrt n}.
\\
&\leq
\frac{(2+\myepsilon)\myepsilon}{e} + \frac{3\delta}{2}.
\end{align*}
We thus have proved that $\pl(\left\lvert R_i(\tomega_\lambda) \right\rvert > \frac{2(2+\myepsilon)\myepsilon}{e} + \frac{3\delta}{2}) \xrightarrow{n \to \infty} 0$ and this concludes the proof since $ \frac{2(2+\myepsilon)\myepsilon}{e} + \frac{3\delta}{2} \xrightarrow{\delta \to 0} 0$.
\end{proof}

\appendix

\section{Intermediate proofs}

We provide here the proofs of Lemma~\ref{lemma:sumexp} and Lemma~\ref{lemma:riemann-lebesgue}.

\subsection{A conditionally convergent series}
\label{subsection:proof_lemma}

We prove here Lemma~\ref{lemma:sumexp}.
Let $k \in \Z$ with $k \notin e\Z$. With $\sumexp[k](x) \coloneqq \sum_{n \in e\Z + k} \frac{\exp 2\ic n x}{n} = 2\sum_{n \in 2e\Z+2k} \frac{\exp \ic nx}{n}$, we want to prove that
\[
\sumexp[k](x) = \frac{\ic}{e}\sum_{\omega \in \Omega_{2e}}\exp(2\ic k\omega)(\omega-\sgnx{x}{\omega}\pi),
\]
for any $x \in (-\pi,\pi) \setminus \Omega_{2e}$, where $\sgnx{x}{\omega} = \sgn(\omega-x)$ and   $\Omega_{2e} = \bigl\{\frac{\ell\pi}{e} : -e < \ell \leq e\bigr\} \subseteq (-\pi,\pi]$. 
 In particular, with $\omega_\ell = \frac{\ell\pi}{e}$ we have $\Omega_{2e} = \{\omega_\ell\}_{-e < \ell \leq e}$ and the above equality shows that $\sumexp[k]$ is piecewise constant on each interval $\bigl(\omega_\ell,\omega_{\ell+1}\bigr)$ for $-e \leq \ell < e$.

We denote by $\ln : \mathbb{C} \setminus \R_- \to \{z \in \mathbb{C} : -\pi < \arg z < \pi\}$ the principal value of the logarithm.  Recall that for any $z \in \mathbb{C}$ with $|z| \leq 1$ and $z \neq 1$ we have
\[
-\ln(1-z) = \sum_{n = 1}^{+\infty} \frac{z^n}{n}.
\]
As in the proof of Lemma~\ref{lemma:sum_Jek+s}, we deduce that if moreover $z \notin \mu_{2e}$ then, recalling~\eqref{equation:sum_root_unity},
\begin{align*}
-\sum_{\zeta \in \mu_{2e}} \zeta^{-2k} \ln(1-\zeta z)
&=
\sum_{n =1}^{+\infty} \frac{z^n}{n} \sum_{\zeta \in \mu_{2e}} \zeta^{n-2k}
\\
&=
2e\sum_{\substack{n = 1 \\ n \in 2e\Z + 2k}}^{+\infty} \frac{z^n}{n}.
\end{align*}
We deduce that
\begin{align*}
\sum_{n \in 2e\Z + 2k} \frac{z^n}{n}
&=
\sum_{\substack{n = 1 \\ n \in 2e\Z + 2k}}^{+\infty} \frac{z^n}{n}
+ \sum_{\substack{n = -\infty \\ n \in 2e\Z+2k}}^{-1} \frac{z^n}{n}
\\
&=
\sum_{\substack{n = 1 \\ n \in 2e\Z + 2k}}^{+\infty} \frac{z^n}{n}
-
\sum_{\substack{n = 1 \\ n \in 2e\Z - 2k}}^{+\infty} \frac{z^{-n}}{n}
\\
&=
-\frac{1}{2e} \sum_{\zeta \in \mu_{2e}} \zeta^{-2k} \ln(1-\zeta z) + \frac{1}{2e}\sum_{\zeta \in \mu_{2e}} \zeta^{2k}\ln\bigl(1-\zeta z^{-1}\bigr)
\\
&=
\frac{1}{2e}\sum_{\zeta \in \mu_{2e}} \zeta^{2k} \left[ \ln\bigl(1-\zeta z^{-1}\bigr) - \ln\bigl(1-\zeta^{-1} z\bigr)\right].
\end{align*}
Now if $|z| = 1$, we have $\arg\bigl(1-\zeta z^{-1}\bigr),\arg\bigl(1-\zeta^{-1} z\bigr) \in (-\pi/2,\pi/2)$  thus
\[
-\pi < \arg\bigl(1-\zeta z^{-1}\bigr) - \arg\bigl(1-\zeta^{-1} z\bigr) < \pi.
\]
We deduce that
\[
\ln\bigl(1-\zeta z^{-1}\bigr) - \ln\bigl(1-\zeta^{-1} z\bigr) = \ln \frac{1-\zeta z^{-1}}{1-\zeta^{-1} z}.
\]
We have $\frac{1-\zeta z^{-1}}{1-\zeta^{-1} z} = -\zeta z^{-1}$ thus we obtain
\[
\sum_{n \in 2e\Z+2k} \frac{z^n}{n} =\frac{1}{2e} \sum_{\zeta \in \mu_{2e}} \zeta^{2k} \ln \left( - \zeta z^{-1}\right),
\]
for $|z| = 1$ and $z \notin \mu_{2e}$. 
With $z = \exp \ic x$ with $x \in (-\pi,\pi) \setminus \Omega_{2e}$, we obtain that, with $\sgnx{x}{\omega} \coloneqq \sgn(\omega-x)$,
\begin{align*}
\sumexp[k](x)
&=
2\sum_{n \in 2e\Z + 2k} \frac{\exp \ic n x}{n}
\\
&=
\frac{1}{e}\sum_{\omega \in \Omega_{2e}} \exp \bigl(2\ic k\omega\bigr) \ln \left(-\exp\ic (\omega - x)\right)
\\
&=
\frac{1}{e}\sum_{\omega \in \Omega_{2e}} \exp \bigl(2\ic k\omega\bigr) \ln \Bigl[\exp\ic (\omega - x - \sgnx{x}{\omega}\pi)\Bigr].
\end{align*}
We have $x \in (-\pi,\pi)$ and $\omega \in (-\pi,\pi]$ thus $\omega-x-\sgnx{x}{\omega}\pi \in (-\pi,\pi)$ thus
\[
\ln\Bigl[\exp\ic (\omega - x - \sgnx{x}{\omega}\pi)\Bigr] = \ic(\omega - x - \sgnx{x}{\omega}\pi).
\]
We deduce that
\begin{align*}
\sumexp[k](x)
&=
\frac{\ic}{e}\sum_{\omega \in \Omega_{2e}}\exp(2\ic k\omega)(\omega-x-\sgnx{x}{\omega}\pi)
\\
&=
\frac{\ic}{e}\sum_{\omega \in \Omega_{2e}} \exp(2\ic k\omega)(\omega - \sgnx{x}{\omega}\pi) - \frac{\ic x}{e}\sum_{\omega \in \Omega_{2e}} \exp(2\ic k \omega),
\end{align*}
which gives the desired result since
\[
\sum_{\omega \in \Omega_{2e}} \exp(2\ic k \omega)
=
\sum_{\zeta \in \mu_{2e}} \zeta^{2k} = 0,
\]
since $e \nmid k$.

\begin{remark}
If $k \in e\Z$, then the sum $\sumexpzero(x) \coloneqq \sum_{n \in e\Z\setminus\{0\}} \frac{\exp 2\ic nx}{n}$ is not piecewise constant. Indeed, we have
\[
\sumexpzero(x) + \sum_{k = 1}^{e-1} \sumexp[k](x) = \sum_{n \neq 0} \frac{\exp 2\ic nx}{n} = 2\sum_{n \geq 1}\frac{\sin 2nx}{n} = 2S(2x),
\]
where $S(y)\coloneqq \sum_{n \geq 1} \frac{\sin ny}{n}$. By a standard equality we have $S(y) = \frac{\pi - y}{2}$ for all $y \in (0,2\pi)$ (see, for instance, \cite[1.441.1]{grad}), thus by Lemma~\ref{lemma:sumexp} we know that $\sumexpzero$ is not piecewise constant  (but piecewise affine). We can also directly compute $\sumexpzero$ since
\[
\sumexpzero(x)= \frac{1}{e}\sum_{n \neq 0} \frac{\sin 2enx}{n} = \frac{2}{e}S(2ex).
\]
\end{remark}

\subsection{Riemann--Lebesgue lemma}
\label{subsection:riemann-lebesgue}

We prove here Lemma~\ref{lemma:riemann-lebesgue}.
Let $f$ be integrable on $(a,b)$ and $\phi$ be continuously differentiable on $[a,b]$. We want to prove that if $\phi$ vanishes at a finite number of points then
\[
\int_a^b f(t) \exp \ic\bigl(x\phi(t)\bigr) dt \xrightarrow{x\to+\infty} 0.
\]

If $\phi'$ does not vanish in $[a,b]$ the result reduces to the classical Riemann--Lebesgue lemma after the variable change $u = \phi(t)$. Thus, by additivity it suffices to consider the case where $\phi$ vanishes only at $a$. Let $0 < \epsilon \ll 1$ and write
\[
\int_a^b f(t) \exp \ic\bigl(x\phi(t)\bigr)dt
=
\int_a^{a+\epsilon} f(t) \exp \ic\bigl(x\phi(t)\bigr)dt
+
\int_{a+\epsilon}^b f(t)\exp \ic\bigl(x\phi(t)\bigr)dt.
\]
The first integral is bounded in module by the integral of $|f|$ over $(a,a+\epsilon)$, and the second one goes to zero as $x \to +\infty$ by the  first case since $\phi'$ does not vanish in $[a+\epsilon,b]$. We conclude that for any $\epsilon > 0$ we have
\[
\limsup_{x \to +\infty} \left\lvert\int_a^b f(t) \exp \ic\bigl(x\phi(t)\bigr)dt\right\rvert \leq \int_a^{a+\epsilon} |f(t)| dt,
\]
and now the quantity in the right-hand side goes to zero as $\epsilon \to 0$ since $f$ is integrable.

\end{document}